\documentclass[10pt]{amsart}
\input epsf

\usepackage{amsfonts,amsthm,amsmath,amssymb,latexsym}
\usepackage{graphicx,mathrsfs,caption,epstopdf}
\usepackage{color}
\usepackage[all]{xy}
\usepackage{wrapfig}
\usepackage{latexsym,epsfig,amscd,graphicx}
\numberwithin{figure}{section}
\numberwithin{equation}{section}
\usepackage{hyperref}
\theoremstyle{definition}
\theoremstyle{plain}
\newtheorem{definition}{Definition}[section]

\newtheorem{remark}[definition]{Remark}

\newtheorem{proposition}[definition]{Proposition}
\newtheorem{theorem}[definition]{Theorem}
\newtheorem{corollary}[definition]{Corollary}
\newtheorem{conjecture}[definition]{Conjecture}
\newtheorem{lemma}[definition]{Lemma}
\hypersetup{colorlinks=true,linkcolor=blue,citecolor=blue}
\newcommand{\Mod}{\mathrm{Mod}}
\newcommand{\eps}{\varepsilon}
\newcommand{\calD}{\mathcal{D}}

\newcommand{\calK}{\mathcal{K}}

\newcommand{\calH}{\mathcal{H}}

\newcommand{\calR}{\mathcal{R}}

\newcommand{\la}{\lambda}

\newcommand{\G}{\Gamma}
\newcommand{\g}{\gamma}
\newcommand{\D}{\Delta}

\renewcommand{\S}{\mathscr{S}}

\newcommand{\dist}{\mathrm{dist}}
\newcommand{\diam}{\mathrm{diam}}
\renewcommand{\mod}{\mathrm{mod}}

\DeclareMathAlphabet{\mathpzc}{OT1}{pzc}{m}{it}
\title[QS Embeddings of Slit Carpets]{Quasisymmetric Embeddings of Slit Sierpi\'nski Carpets}

\author{Hrant Hakobyan}
\address{Department of Mathematics, Kansas State University, Manhattan, KS, 66506-2602}
\thanks{H.~H. was partially supported by Simons Foundation Collaboration Grant, award ID: 638572.}
\email{hakobyan@math.ksu.edu}

\author{Wen-bo Li}
\address{Department of Mathematics, University of Toronto, 40 St. George Street, Toronto, Ontario, Canada M5S 2E4}
\email{wenboli@math.toronto.edu}

\date{\today}
\begin{document}

\maketitle

\renewcommand{\thefootnote}{\fnsymbol{footnote}} 
\renewcommand{\thefootnote}{\arabic{footnote}} 

\renewcommand{\thefootnote}{\fnsymbol{footnote}} 
\renewcommand{\thefootnote}{\arabic{footnote}} 

\renewcommand{\thefootnote}{\fnsymbol{footnote}} 
\renewcommand{\thefootnote}{\arabic{footnote}} 

\begin{abstract}
We study the problem of quasisymmetrically embedding  spaces homeomorphic to the  Sierpi\'nski carpet into the plane.  In the case of so called dyadic slit carpets, several characterizations are obtained.  One characterization is in terms of a Transboundary Loewner Property (TLP) which is  a transboundary analogue of the Loewner property  of Heinonen and Koskela. We show that a dyadic slit carpet 
can be quasisymmetrically embedded into the plane if and only if it is TLP. Moreover, every dyadic slit carpet $X$ can be associated to a ``pillowcase sphere"  $\widehat{X}$ which is  a metric space homeomorphic to the sphere $\mathbb{S}^2$. We show that $X$ quasisymmetrically embeds into the plane if and only if $\widehat{X}$ is quasisymmetric to $\mathbb{S}^2$ if and only if $\widehat{X}$ is Ahlfors $2$-regular. 
\end{abstract}

\tableofcontents


\section{Introduction}

\subsection{Quasisymmetric planarity of metric carpets}
A metric space is said to be a \emph{metric (Sierpi\'nski) carpet} if it is homeomorphic to the classical Sierpi\'nski carpet $\mathbb{S}_{1/3}$, see Fig. \ref{standardcarpet}.  The study of quasisymmetric  geometry of metric carpets has received much attention recently, see e.g., \cite{Bonk,BonkKleinerMerenkov,BonkMerenkov,KleinerICM, Haisinsky,Merenkov,Merenkov Wildrick,MTW}. An important problem in this direction is to understand when a metric carpet admits a quasisymmetric embedding  into the complex plane $\mathbb{C}$, or is \emph{quasisymetrically (or QS) planar}. 
\begin{wrapfigure}{r}{0.3\textwidth}
    \includegraphics[width=0.33\textwidth]{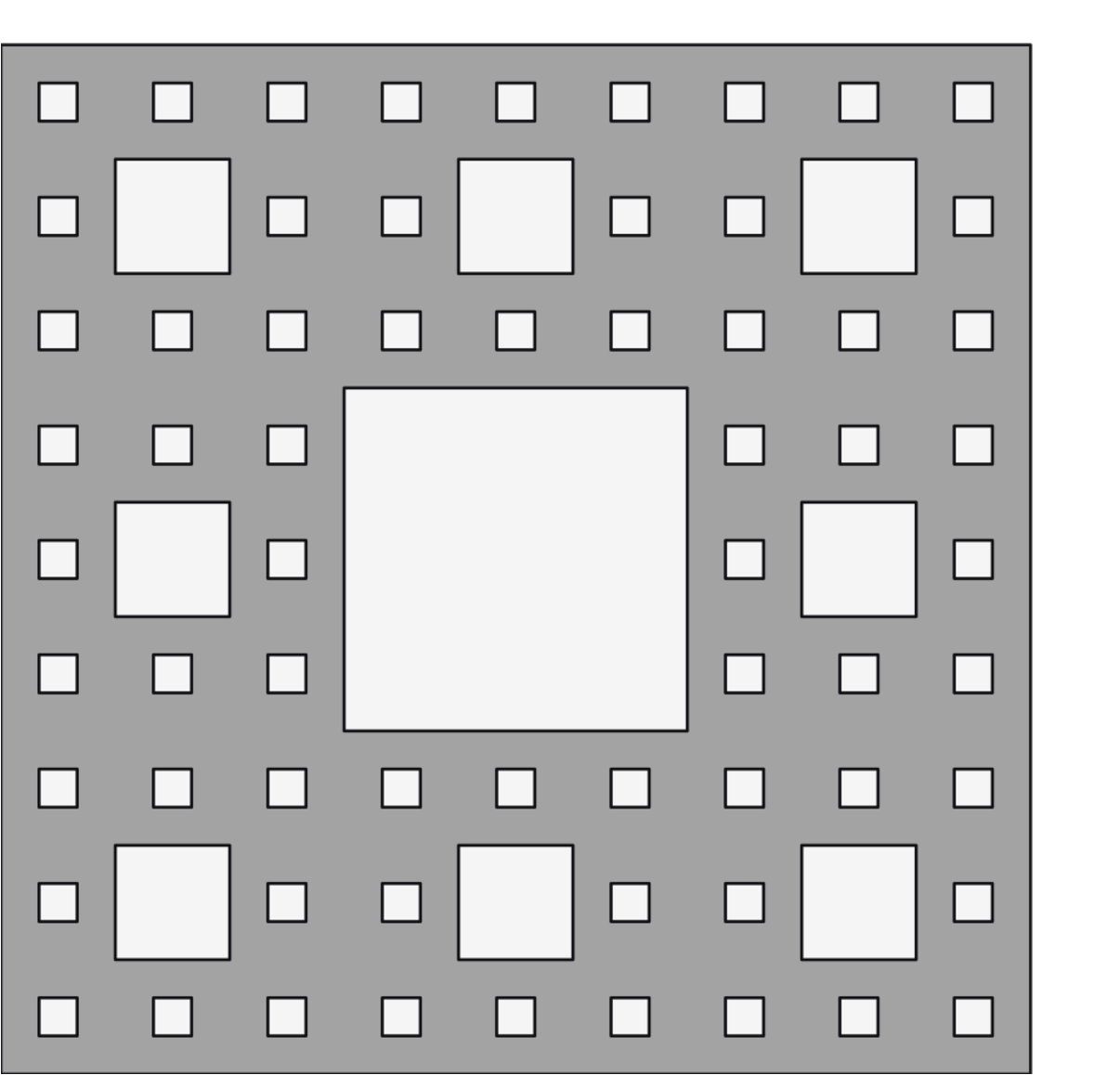}
  \caption{First three steps in the construction of the classical Sierpi\'nski carpet $\mathbb{S}_{1/3}$.}
  \label{standardcarpet}
\end{wrapfigure}
In this paper we give several characterizations of QS planarity for a class of metric carpets called \emph{dyadic slit carpets}, see Theorems \ref{thm:slit-carpets-TLP}, \ref{thm:surface}, and \ref{thm:main}.

The study of QS planarity of Sierpi\'nski carpets is partly motivated by Kapovich-Kleiner conjecture in geometric group theory. A metric carpet $X$ is called a \emph{group carpet} if it occurs as the boundary at infinity $\partial_{\infty}G$ of a Gromov hyperbolic group $G$. The above mentioned conjecture suggests that every group carpet is quasisymmetric to a \emph{round carpet}, that is metric carpet $X\subset\mathbb{C}$ such that every complimentary component of $X$ is a round disk.
 %
%
%
%

%

In a recent breakthrough Bonk \cite{Bonk} proved that a  metric carpet $X\subset\mathbb{C}$ is quasisymmetric to a round carpet provided some natural conditions are satisfied (see (\ref{def:uni-rel-sep}) and Section \ref{Section:background} for the terminology).

\begin{theorem}[Bonk's Theorem]\label{thm:Bonk}
Let $X\subset\mathbb{C}$ be a metric carpet whose peripheral circles are uniformly relatively separated uniform quasicircles. Then there is a quasisymmetric mapping $f:\mathbb{C}\to\mathbb{C}$ such that $Y=f(X)$ is a round carpet.
\end{theorem}

If $X$ is a group carpet then all the conditions of Theorem \ref{thm:Bonk} are satisfied except for planarity, i.e. $X\subset \mathbb{C}$, see \cite[Proposition 1.4]{Bonk}. Therefore, Kapovich-Kleiner conjecture would be true if every carpet boundary were QS planar. 



\subsection{QS planarity and Transboundary Lowner Property}
In \cite{Merenkov Wildrick} Merenkov and Wildrick showed that there is a metric carpet $X$, which is not QS planar (and thus is not quasisymmetric to a round carpet) even though all the conditions in Bonk's theorem hold, with the exception of $X$ being a subset of $\mathbb{C}$. Thus, one may wonder what additional conditions would imply QS planarity. 

In this paper we identify conditions which guarantee QS planarity of certain non-planar metric carpets which are obtained as limits of finitely connected planar domains (equipped with the inner path metric).
One such condition is in terms of the reciprocal of Schramm's transboundary extremal length \cite{Schramm}, nowadays called the transboundary modulus, and resembles the well-known Loewner property of Heinonen and Koskela \cite{HK:Acta}.

 \begin{definition}
Let $\Omega\subset\mathbb{C}$ be a simply connected domain, $\calK=\{k_i\}_{i=1}^m$ a family of disjoint, non-degenerate continua in $\Omega$. Given a function $\Psi:(0,\infty)\to(0,\infty)$, we say that the finitely connected domain $W=\Omega\setminus\cup_{i=1}^m k_i$ satisfies the  $(\Psi)$-\textsc{Transboundary Loewner Property}, or is $(\Psi)$-$\mathrm{TLP}$, with respect to the inner path metric $\delta_W$ on $W$,  if for every pair of disjoint continua $E,F\subset W$  we have  
\begin{eqnarray}\label{TLP:intro}
\Mod_{\Omega,\calK} \Gamma(E,F;\Omega) \geq \Psi \left(\Delta(E,F)\right).
\end{eqnarray}
\end{definition}

In (\ref{TLP:intro}) above $\Mod_{\Omega,\calK} \Gamma(E,F;\Omega)$ is the transboundary modulus of the path family connecting $E$ and $F$ in $\Omega$ with respect to $\calK$, and $\Delta(E,F)$ is the relative distance between $E$ and $F$ \textit{in the inner metric $\delta_W$ of $W$}. 
We refer to Sections \ref{Section:Modulus} and  \ref{Section:TLP} for the definition of transboundary modulus  and  further discussion of  TLP. In (\ref{TLP:intro}) above we used notation $\Mod_{\Omega,\calK}(\G)$ for the transboundary modulus of a path family $\G$ following Bonk \cite{Bonk} and Merenkov \cite{Merenkov-rel-Schottky}. 

\begin{remark}
\rm{Transboundary Loewner property may of course be considered not only for the inner metric $\delta_W$ but for any metric on $W$ (as long as the left-hand side of (\ref{TLP:intro}) is well-defined and is positive). Since we are mostly concerned with the applications to slit carpets, which are obtained as limits of domains equipped with the inner metrics, we do not consider the TLP with respect to any other metrics in this paper.} 
\end{remark}

The importance of transboundary Loewner property for us is that it is a necessary condition for QS planarity for multiply connected domains, see Theorem \ref{thm:finitely-connected-embedding}. For instance, we have the following corollary of Theorem \ref{thm:finitely-connected-embedding}. Recall, that a domain $\Omega\subset \mathbb{C}$ is a \emph{circle domain} if every boundary component of $\Omega$ is a round circle or a point.

\begin{theorem}\label{thm:finitely-connected-TLP-intro}
Suppose $W\subset\mathbb{C}$ is a finitely connected domain. If there is an $\eta$-quasisymmetric map $f:(W,\delta_W)\to\mathbb{C}$ such that $f(W)$ is a circle domain then $W$ is $\Psi$-TLP, where $\Psi$ depends only on $\eta$ (and not on the number or size of the boundary components of $W$).
\end{theorem}

We say that a metric carpet is TLP if it is a limit of a sequence of uniformly TLP domains (appropriately defined). See Subsection \ref{Secion:TLP-intro} and in particular Definitions \ref{def:TLP-domains} and \ref{def:TLP-inverse-limits} for the precise definition of TLP carpets. 

One of our main results is that TLP is often also sufficient for QS planarity. In particular, we obtain the following characterization of QS planarity in the class of dyadic slit carpets (see the construction before Theorem \ref{thm:main}). 

\begin{theorem}\label{thm:slit-carpets-TLP}
 A dyadic slit carpet $X$ is QS planar if and only if it is TLP.
\end{theorem}

A key feature here is that Transboundary Loewner Property is an intrinsic \textit{quasisymmetrically invariant} condition. To our knowledge Theorem \ref{thm:slit-carpets-TLP} is the first instance when such a condition characterizes spaces which are QS planar. 

In \cite{Merenkov Wildrick} sufficient conditions for embedding a metric space homeomorphic to a planar domain into the plane were obtained. These conditions are not quasisymmetrically invariant and the authors explicitly asked for such a condition. One may observe that in the case of dyadic slit carpets TLP is equivalent to the conditions in \cite{Merenkov Wildrick}. However, there are many spaces for which TLP holds while the conditions in \cite{Merenkov Wildrick} do not.

To prove Theorem \ref{thm:slit-carpets-TLP} we obtain several other characterizations of QS planarity for dyadic slit carpets which we describe next.



%
%
%

\subsection{Pillowcase spheres}
A general method of obtaining sufficient conditions for a topologically planar metric space $X$ to admit a quasisymmetric embedding into the plane is by means of the celebrated uniformization theorem of Bonk
 and Kleiner \cite{Bonk Kleiner}. The latter states that  a metric space that is homeomorphic to the $2$-sphere $\mathbb{S}^2$ is  in fact quasisymmetric to $\mathbb{S}^2$ (equipped with the spherical metric) provided it is Ahlfors $2$-regular and is linearly locally connected. Here, linear local connectivity (or LLC) is a quasisymmetrically invariant condition that is necessary for a space to be quasisymmetric to $\mathbb{S}^2$, see Section \ref{Section:embedding}.  Also, a metric measure space $(X,d,\mu)$ is Ahlfors $Q$-regular if there is a constant $C\geq 1$ such that for every ball $B(x,r)\subset X$ the following inequalities hold:   
 \begin{align}
 C^{-1}r^Q\leq \mu(B(x,r))\leq C r^Q.
 \end{align}   
 
In view of Bonk and Kleiner's theorem,  a metric carpet $X$ can be quasisymmetrically embedded into $\mathbb{R}^2$ if it is possible to construct a metric sphere $\widehat{X}$ containing $X$ which is LLC and Ahlfors $2$-regular. 
For a slit carpet $X$ there is a natural way of constructing a metric sphere $\widehat{X}\supset X$ by gluing in topological disks, or ``pillowcases", to the slits of $X$, see Figure \ref{figure:slit_domain} and Section \ref{Section:gluing} and then ``doubling" the resulting topological disk along the boundary square. The resulting ``pillowcase sphere" $\widehat{X}$ is always linearly locally connected. Therefore, $X$ can be quasisymmetrically embedded into the plane if $\widehat{X}$ is Ahlfors $2$-regular. We show that the converse is also true.

\begin{figure}[t]
	\centering
	\includegraphics[height = 1.5in]{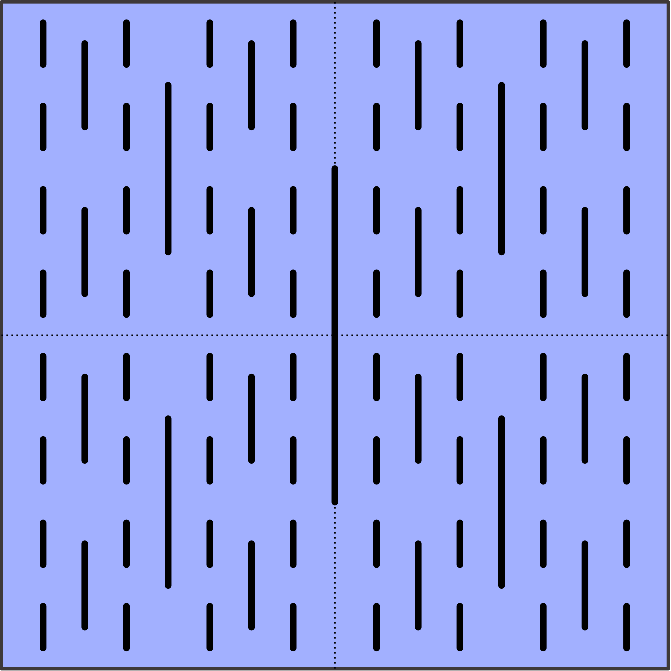}
\qquad
	\includegraphics[height = 1in]{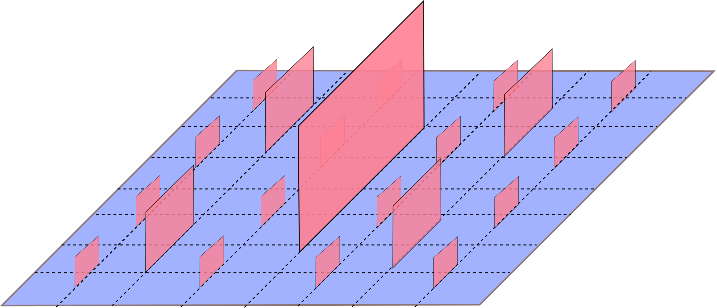}
	\caption{{Slit domain $D$ (left) corresponding to the sequence $(1/2, 1/2, 1/2, 1/2)$ and the corresponding ``pillowcase surface" $\widehat{D}$ (right). Every point of a slit $s\subset \partial D$ which is not an endpoint gives rise to two points in the closure $\overline{D}$ in the inner length metric and therefore topologically $s$ is a circle in $\overline{D}$. Similarly, every square in $\widehat{D}$ attached to a slit of $D$ can be thought of as a disk obtained from two copies of a square identified along three edges (hence, a ``pillowcase") that is glued to $\overline{D}$ along the other two edges. A ``pillowcase  sphere" is obtained by gluing two copies of $\widehat{D}$ along the ``outer square" boundary.}}
\label{figure:slit_domain}
\end{figure}

\begin{theorem}\label{thm:surface}
Let $X$ be a dyadic slit carpet and $\widehat{X}$ a ``pillowcase sphere"
corresponding to $X$. Then the following conditions are equivalent:
\begin{enumerate}
\item $X$ is quasisymmetric to a subset of $\mathbb{R}^2$.
\item $\widehat{X}$ is quasisymmetric to $\mathbb{S}^2$.
\item $\widehat{X}$ is Ahlfors $2$-regular.
\end{enumerate}
\end{theorem}

As explained above the implication $(3)\Rightarrow(2)$ follows from the theorem of Bonk and Kleiner. $(2)\Rightarrow(1)$ follows from the construction of $\widehat{X}$ in Section \ref{Section:gluing}. Thus,  the main novelty of Theorem \ref{thm:surface} is  the implication $(1)\Rightarrow(3)$. 

%
%
%


Next we define the dyadic slit carpets precisely. Let $\mathbf{r}=\{r_i\}_{i=0}^{\infty}$ be a sequence of real numbers such that $r_i\in(0,1), i\geq 0$. Construct a nested sequence of domains $D_i$ in the plane corresponding to $\mathbf{r}$ as follows: Let $D_0$ denote the domain obtained by removing the closed vertical slit (interval) of length $r_0$ centered at $(1/2,1/2)$ from  $(0,1)^2$. Similarly $D_1$ is obtained by removing from $D_0$ the $4$ vertical slits of length $r_1/2$, which are located in the dyadic squares of generation $1$ and whose centers are at the centers of the corresponding squares. Continuing by induction we obtain a sequence of domains $D_{i+1}\subset D_i$ in the unit square $(0,1)^2$. Next, consider the sequence of metric spaces $S_i$, where $S_i$ is the completion of the domain $D_i$ in its inner path metric $d_{D_i}$. It turns out that the spaces $S_i$ converge  (in an appropriate sense) to a metric carpet, which we denote by $\S=\mathscr{S}_{\mathbf{r}}$ and call the \emph{dyadic slit carpet corresponding to $\mathbf{r}$}. The following is the main result of this paper.

\begin{theorem}\label{thm:main}
	Suppose $\mathscr{S}_{\mathbf{r}}$ is a dyadic slit carpet corresponding to a sequence $\mathbf{r}=\{r_i\}_{i=0}^{\infty}$. There is a quasisymmetric embedding of $\mathscr{S}_{\mathbf{r}}$ into the plane if and only if $\mathbf{r} \in \ell^2$.
\end{theorem}

Theorem \ref{thm:main} implies Theorem \ref{thm:surface}. Indeed, as explained above, one only needs to prove the implication $(1)\Rightarrow(3)$. However, if $(1)$ in Theorem  \ref{thm:surface} holds then by Theorem \ref{thm:main} $\mathbf{r}\in\ell^2$. On the other hand, in Section \ref{Section:embedding} it will be shown that if $\mathbf{r}\in\ell^2$ then $\widehat{X}=\widehat{\mathscr{S}}_{\mathbf{r}}$ is Ahlfors $2$-regular, which is condition $(3)$ in Theorem \ref{thm:surface}. Therefore, $(1)\Rightarrow(3)$ in Theorem \ref{thm:surface} holds.

\subsection{Quasisymmetric embedding and weak tangents}
One consequence of Theorem \ref{thm:main} is that a metric carpet may not admit a quasisymmetric embedding into the plane even if locally it ``looks" like $\mathbb{R}^2$.

\begin{corollary}
There is a metric carpet $X$ such that every weak tangent of $X$ is bi-Lipschitz to a subset of $\mathbb{R}^2$ but which cannot be quasisymmetrically embedded into $\mathbb{R}^2$.
\end{corollary}

We refer the reader to \cite{Burago} for the definition and the properties of weak tangents and Gromov-Hausdorff convergence. 

To obtain an example as in the statement above, one may pick a sequence $r_i$ which converges to $0$ but such that $\sum_{i}r_i^2=\infty$. Then, since $r_i\to0$, every weak tangent of $\mathscr{S}_{\mathbf{r}}$ is bi-Lipschitz to a subset of the plane, see \cite{Wenbo}. On the other hand, $\mathscr{S}_{\mathbf{r}}$ does not quasisymmetrically embed in $\mathbb{R}^2$ by Theorem \ref{thm:main}. 

Note that the self-similar slit carpet considered in \cite{Merenkov} looks very ``non-Euclidean" on all scales and in fact its weak tangents do not admit bi-Lipschitz embeddings into any uniformly convex Banach space, see \cite{David-Eriksson}.



%



\subsection{Round carpets}

A family of sets $\{K_i\}_{i=1}^{\infty}$ in a metric space is said to be \emph{uniformly relatively separated} if the relative distances between them are uniformly bounded from below, i.e., if  there is a constant $s>0$  such that for $i\neq j$ the following holds:
\begin{align}\label{def:uni-rel-sep}
 \Delta(K_i,K_j) = \frac{\textrm{dist}(K_i,K_j)}{\min\{\textrm{diam}(K_i), \textrm{diam}(K_j)\}}>s.
\end{align} 

A family of Jordan curves $\{\g_i\}_{i=1}^{\infty}$ in $\mathbb{C}$  is said to be a family of \emph{uniform quasicircles} if there is a $k\geq 1$, such that $\g_i$ is a $k$-quasicircle for every $i\geq 1$ (see (\ref{quasicircle}) in Section \ref{Section:background} for the definition of a $k$-quasicircle). 


By Theorem \ref{thm:main} if there is an embedding of $X$ into the plane then the slits of $X$ are relatively separated. Therefore by Bonk's theorem we obtain the following.

\begin{corollary}\label{cor:round-carpets}
Suppose $\S_{\mathbf{r}}$ is a dyadic slit carpet whose peripheral circles are uniformly relatively separated. Then $\S_{\mathbf{r}}$ is quasisymmetric to a round carpet if and only if $\mathbf{r}\in\ell^2$ if and only if $\S_{\mathbf{r}}$ is quasisymetric to a subset of the plane.
\end{corollary}

All the quasisymmetric images of slit carpets have positive $\calH^2$ measure, cf., \cite{Merenkov,Hakobyan}. Therefore all the round carpets which are quasisymmetric to slit carpets are of positive area. By \cite[Theorem 1.2]{BonkKleinerMerenkov}, for every positive area round carpet in $\mathbb{S}^2$ there are uncountably many conformally distinct round carpets which are all quasisymmetrically equivalent to it. In particular,  for every slit carpet $\mathscr{S}_{\mathbf{r}}$ such that $\mathbf{r}\in \ell^2$ there are uncountably many  round carpets which are not M\"obius images of each other, but which are all quasisymmetrically equivalent to $\mathscr{S}_{\mathbf{r}}$. 

 It would be interesting to know if there are quasisymmetric \emph{self maps} of $\mathscr{S}_{\mathbf{r}}$ which are not isometries. More generally, what can be said about the group of quasisymmetric automorphism of $\mathscr{S}_{\mathbf{r}}$? Is it finite, infinite, uncountable? In \cite{Merenkov} it was shown that every quasisymmetric self map of the slit carpet $\mathscr{S}_{\mathbf{r}}$  corresponding to the constant sequence $r_i=1/2$ is in fact an isometry. 

\subsection{Quasisymmetric embeddings and Loewner carpets} From Theorem \ref{thm:main} and results in \cite{Hakobyan} it follows that property of being quasisymmetrically embeddable into the plane is related to the Loewner property of Heinonen and Koskela, which we recall next.

Suppose $(X,d,\mu)$ is an Ahlfors $Q$-regular metric measure space for some $Q>1$. $X$ is \emph{$Q$-Loewner} if there is a function $\phi:(0,\infty)\to(0,\infty)$ such that  for every pair of continua $E$ and $F$ in $X$ the following holds:
\begin{align*}
\mod_Q \left( \G(E,F;X) \right) \geq \phi(\Delta(E,F)),
\end{align*}
where $\G(E,F;X)$ is the family of curves connecting $E$ and $F$ in $X$ and $\Delta(E,F)$ is the relative distance between $E$ and $F$.

Loewner spaces have been introduced by Heinonen and Koskela in \cite{HK:Acta} and since then have been studied extensively. We will say a metric space is a \emph{Loewner carpet} if it is homeomorphic to the Sierpi\'nski carpet and is $Q$-Loewner for some $Q\geq1$.

Recently, Cheeger and Eriksson-Bique \cite{Cheeger-Eriksson}  showed that every $Q$-Loewner carpet  can be quasisymmetrically embedded into the plane, provided $1<Q<2$. 

On the other hand, a $Q$-Loewner space cannot be quasisymmetrically mapped to a space of Hausdorff dimension less than $Q$ by a theorem of Tyson, see \cite[Theorem, 15.10]{Heinonen}. Therefore, a $Q$-Loewner carpet cannot be embedded into the plane if $Q>2$.

The borderline case of $2$-Loewner carpets is not yet understood completely. However, in the case of dyadic slit carpets we have the following. 

\begin{corollary}\label{QS-embed-and-PI}
Every $2$-Loewner dyadic slit carpet $\mathscr{S}_{\mathbf{r}}$ admits a quasisymmetric embedding into $\mathbb{R}^2$.
\end{corollary}
\begin{proof}
Suppose $\mathscr{S}_{\mathbf{r}}$ cannot be quasisymmetrically embedded in the plane. By Theorem \ref{thm:main} it follows that $\mathbf{r}\notin \ell^2$. In \cite[Theorem 12.3]{Hakobyan} it was shown that if $\mathbf{r}\notin\ell^2$ then the $2$-modulus of curves connecting the right and left ``vertical edges" in $\mathscr{S}_{\mathbf{r}}$ vanishes. Hence, $\mathscr{S}_{\mathbf{r}}$ is not $2$-Loewner.
\end{proof}

In view of Corollary \ref{QS-embed-and-PI} and the results in \cite{MTW} on non-self-similar square carpets, it is natural to expect that $\mathscr{S}_{\mathbf{r}}$  can be quasisymmetrically embedded in $\mathbb{R}^2$ if and only if it is $2$-Loewner (or even admits a $p$-Poincar\'e inequality for all $p>1$). To prove this one needs to show that if $\mathbf{r}\in\ell^2$ then $\mathscr{S}_{\mathbf{r}}$ is $2$-Loewner, i.e., there are uniform lower bounds on $\mod_2\left( \G(E,F;X) \right)$ for all compact connected subsets $E$ and $F$ in $\mathscr{S}_{\mathbf{r}}$. We do not establish such bounds in this paper.


%
%
%

\subsection{General slit carpets and domains} General slit carpets corresponding to a sequence of slits $\{v_i\}_{i=1}^{\infty}$ in a Jordan domain (or a quasidisk) $\Omega\subset\mathbb{C}$ can be constructed the same way as the dyadic slit carpets. Namely, if the closures of the slit domains $D_i=R\setminus \cup_{j=1}^i v_i$ in the inner metric converge to a metric space $X$ that is homeomorphic to the Sierpi\'nski carpet we call $X$ a (general) \emph{slit carpet}.

It is natural to ask what can be said about quasisymmetric embeddability of general slit carpets. Theorem \ref{thm:surface} suggests that one may answer this in terms of the ``pillowcase" surface $\widehat{X}$, which can be constructed for every slit carpet ${X}$ just like in Section \ref{Section:gluing}. Furthermore, the conditions of being quasisymmetric to round carpets or having Loewner property also can be formulated for any slit carpet. However, it is not hard to see that the slits need to be uniformly relatively separated in order for either of these conditions to hold. Therefore, we believe the following is true.
 
\begin{conjecture}\label{conjecture}
Let $X$ be a general slit carpet constructed in a quasidisk $\Omega\subset \mathbb{C}$. Then the following conditions are equivalent: 
\begin{enumerate}
\item $X$ is quasisymmetric to a subset of $\mathbb{R}^2$.
\item $\widehat{X}$ is quasisymmetric to $\mathbb{S}^2$.
\item $\widehat{X}$ is Ahlfors $2$-regular.
\item $X$ is TLP
\end{enumerate}
Moreover, if the peripheral circles of $X$ are uniformly relatively separated then
\begin{itemize}
\item[$(5)$] $X$ is quasisymmetric to a round carpet.
\item[$(6)$] $X$ is $2$-Loewner.
\end{itemize}
\end{conjecture}

By Theorem \ref{thm:surface} and Corollaries \ref{cor:round-carpets} and  \ref{QS-embed-and-PI}  the following implications hold in the case of dyadic slit carpets (some are known without the assumption of relative separation)
\begin{align*}
(1)\Leftrightarrow (2)\Leftrightarrow (3)  \Leftrightarrow(4) \Leftrightarrow(5)\Leftarrow(6). \tag{Dyadic slit carpets}
\end{align*}
On the other hand, for general slit carpets only the implications below are known (assuming relative separation where appropriate),
\begin{align*}
(1)\Leftarrow (2)\Leftarrow (3), \,\, (1)\Leftrightarrow (5). \tag{General slit carpets}
\end{align*}

One may also formulate the analogue of Conjecture \ref{conjecture} for slit domains. Note that for such a domain $X$ one can construct the corresponding pillowcase surface $\widehat{X}$. Also, condition $(5)$ above should be reformulated for domains as follows:
\begin{itemize}
\item[$(5')$] $X$ is quasisymmetric to a circle domain.
\end{itemize}
Here a circle domain is one whose every boundary component is either a round circle or a point. In this generality $(5')$ and $(6)$ are not equivalent, since there are circle domains which are not Loewner (e.g., there are Cantor sets whose complement is not Loewner). However, if the collection of point boundary components is small enough (e.g., countable), then Loewner condition may still be equivalent to $(4')$. Thus the analogue of Conjecture \ref{conjecture}  for countably connected slit domains can be formulated as follows.

\begin{conjecture}[\textsc{QS Koebe Uniformization for slit domains}]
Let $X\subset \mathbb{C}$ be a countably connected slit domain with uniformly relatively separated boundary components. Then the equivalences below hold
\begin{align*}
(1)\Leftrightarrow (2)\Leftrightarrow (3)  \Leftrightarrow(4) \Leftrightarrow(5') \Leftrightarrow(6). \tag{General slit domains}
\end{align*} 
\end{conjecture}

Quasisymmetric uniformization by circle domains was considered in \cite{Merenkov Wildrick} where general sufficient conditions were obtained. However these conditions do not allow to conclude $(6) \Rightarrow (5')$ above. 



It would be interesting to know if TLP characterizes QS planarity for more general metric carpets. We say that $X$ is a \emph{planar inverse limit} if it is an inverse limit of finitely connected planar domains $W_i\subset \mathbb{C}$ equipped with the inner path metric $d_{W_i}$ (or any conformal metric $d_i$), see Subsection \ref{Section:limits-of-domains}. We believe the following generalization of Bonk's theorem is true.

\begin{conjecture}
Suppose $X$ is a planar inverse limit carpet whose peripheral circles are uniformly relatively separated uniform quasicircles. Then $X$ is QS planar if and only if it is TLP.
\end{conjecture}

\subsection{Outline of the paper} 


This paper is organized as follows: In Section \ref{Section:background} we provide the necessary background material. In Section \ref{Section:Modulus} we define transboundary modulus and study some of its basic properties. 
In Section \ref{Section:TLP} we define the transboundary Loewner property and prove Theorem \ref{thm:finitely-connected-embedding}, which implies Theorem \ref{thm:finitely-connected-TLP-intro}. In Section \ref{Section:slit-carpets} we define the dyadic slit carpets and list some of their basic properties. Sections \ref{Section:necessary-condition} and \ref{Section:non-embedding} are devoted to the proof of necessity in Theorem \ref{thm:main}, i.e.,  if $\mathbf{r}\notin\ell^2$ then $\mathscr{S}_{\mathbf{r}}$ does not quasisymmetrically embed into the plane. In Section \ref{Section:embedding} we construct the ``pillowcase surface" $\widehat{\mathscr{S}}$ and prove the sufficiency in Theorem \ref{thm:main}. Finally, Theorem \ref{thm:slit-carpets-TLP} is proved in Section \ref{Section:proof}.

\section{Background}
\label{Section:background}

\subsection{Notations and Definitions}
Given a metric space $(X, d)$, a point $x \in X$ and $r>0$, we denote by $B(x,r)$ the open ball of radius $r$ centered at $x$, i.e., $B(x,r) = \{y \in X: d(x,y) < r \}.$ For a ball $B=B(x,r)\subset X$ and $\lambda>0$ we let $\lambda B = B(x,\lambda r)$.

The closed unit disk and its boundary circle in the Euclidean plane $\mathbb{R}^2$ will be denoted by $\mathbb{D}$ and $\partial \mathbb{D}$, respectively. 

If $E \subset X$, then the closure, interior and topological boundary of $E$ will be denoted by $\overline{E}$, $\mathrm{int}(E)$, and $\partial E$, respectively. The diameter of $E$ in $X$ and the distance between subsets $E$ and $F$ of $X$ are defined as follows:
\begin{align*}
  \textrm{diam}(E)&=\sup\{d(x,y) : x,y \in E\},\\
  \textrm{dist}(E,F)& = \inf\{d(x,y): x \in E, y \in F\}.
\end{align*}
Sometimes we will write $\diam_X(E)$ and $\dist_X(E,F)$  to emphasize the metric with respect to which these quantities are being calculated.

If $\textrm{diam}(E) > 0$ and $\textrm{diam}(F) > 0$, the \emph{relative distance between $E$ and $F$} is
\begin{align}
 \Delta(E,F) = \frac{\textrm{dist}(E,F)}{\min\{\textrm{diam}(E), \textrm{diam}(F)\}}.
\end{align}

Let $I$ be a finite or countable indexing set. A family $\calK=\{K_i\}_{i \in I}$ of subsets of $X$ is called \emph{$s$-relatively separated} for $s > 0$ if $\Delta(K_i,K_j) \geq s$ for every $i,j \in I, i \neq j$. The sets in $\mathcal{K}$ are said to be \emph{uniformly relatively separated} if they are $s$-relatively separated for some $s > 0$.

Everywhere in this paper we will denote by $\mathcal{H}^2$ the \emph{normalized Hausdorff $2$-measure} on $X$. More specifically, $\calH^2(E)=\lim_{\epsilon\to0} \calH^2_{\epsilon}(E)$, where
\begin{align*}
\calH^2_{\epsilon}(E) = \inf \left\{ \frac{\pi}{2}\sum_{i=1}^{\infty} r_i^2 : E\subset \bigcup_{i=1}^{\infty} B(x,r_i), \, r_i\leq \epsilon\right\}  
\end{align*}
This choice is made so that $\calH^2$ coincides with the $2$ dimensional Lebesgue measure $\mathcal{L}^2$ for subsets of the plane and for spaces isometric to these.

A metric space $(X, d)$ is said to be \emph{Ahlfors $2$-regular}, if there is a constant $C \geq 1$ such that
\begin{align}\label{def-ineq:Ahlfors}
  {r^2}/{C} \leq \calH^2(B(p,r)) \leq C \cdot r^2
\end{align}
for any $p \in X$ and $0 < r \leq \textrm{diam}(X)$. The constant $C$ in (\ref{def-ineq:Ahlfors}) will be called the \emph{Ahlfors regularity constant} of $X$. Sometimes, when the constants are not important, the upper and lower estimates of $\calH^2(B(x,r))$ in (\ref{def-ineq:Ahlfors}) will be written as
\begin{align*}
\calH^2(B(x,r))\lesssim r^2 \mbox{ and } \calH^2(B(x,r))\gtrsim r^2,
\end{align*}
respectively, while if both inequalities hold we may write $\calH^2(B(x,r))\asymp r^2$, instead of (\ref{def-ineq:Ahlfors}).

\subsection{Quasiconformal and quasisymmetric mappings}\label{Section:QS-definition}
Here we define the various classes of mappings we are going to work with and refer to \cite{Ahlfors}, \cite{Heinonen} and \cite{Vaisala} for further details and properties of these maps.

Let $f: X \to Y$ be a homeomorphism between two metric spaces $(X, d_X)$ and $(Y, d_Y)$. For a point $x \in X$ and $r > 0$, we define the \emph{linear dilatation of $f$ at $x$} as
\begin{align}
    H_f(x)=\limsup_{r \to 0}\frac{L_f(x,r)}{l_f(x,r)},
\end{align}
where
\begin{align*}
  L_f(x,r)&= \sup_{y}  \{ d_Y(f(x),f(y)) \, | \, d_X(x,y)\leq r\},\\
  l_f(x,r)&= \inf_{y}   \{ d_Y(f(x),f(y)) \, | \, d_X(x,y)\geq r\}.
\end{align*}

We say that a homeomorphism $f: X\to Y$ is \emph{(metrically) $H$-quasiconformal (or $H$-qc)} if
\begin{align*}
  \sup_{x \in X}H_f(x) \leq H
\end{align*}
for some $1\leq H <\infty$.
A map is quasiconformal if it is $H$-quasiconformal for some $H$.


A homeomorphism $f: X \to Y$ is called \emph{$\eta$-quasisymmetric}, where $\eta:[0, \infty) \to [0, \infty)$ is a given homeomorphism, if
\begin{align}\label{def:QS}
\frac{d_Y(f(x),f(y))}{d_Y(f(x),f(z))} \leq \eta \left(\frac{d_X(x,y)}{d_X(x,z)} \right)
\end{align}
for all $x, y, z \in X$ with $x \neq z$. The map $f$ is called \emph{quasisymmetric} if it is $\eta$-quasisymmetric for some \emph{distortion function} $\eta$.



Here are some useful properties of quasisymmetric maps, which will be used repeatedly in the paper, see \cite{Heinonen}.

\begin{lemma}\label{lemma:QSproperties} 
Suppose $f:X\to Y$ and $g:Y\to Z$ are $\eta$ and $\eta'$-quasisymmetric mappings, respectively.
\begin{itemize}
\item[(1).] The composition $f\circ g:X\to Z$ is an $\eta' \circ \eta$-quasisymmetric map.
\item[(2).] The inverse $f^{-1}:Y\to X$ is a $\theta$-quasisymmetric map, where $\theta(t) = 1/\eta(1/t)$.
\item[(3).] If $A$ and $B$ are subsets of $X$ and $A\subset B$, then
\begin{align}\label{diam}
\frac{1}{2\eta\left( \frac{\diam(B)}{\diam(A)} \right)} \leq \frac{\diam(f(A))}{\diam(f(B))} \leq \eta\left( \frac{2\diam(A)}{\diam(B)} \right).
\end{align}
\end{itemize}
\end{lemma}


%
%
%

The following result is elementary but will be useful for us.

\begin{lemma}\label{lemma:rel-dist-distortion}
Let $A,B$ be compact subsets of $X$ and $f:X\to Y$ a quasisymmetric mapping. Then
\begin{align}\label{ineq:rel-dist-distortion}
\frac{1}{2\eta(\Delta(A,B)^{-1})}\leq \D(f(A),f(B)) \leq \eta(2\D(A,B)).
\end{align} 
\end{lemma}

\begin{proof} Suppose $\diam f(A)\leq \diam f(B)$. 
Let $x,z\in A$ and $y\in B$, be such that $\dist(x,y) = \dist(A,B)$ and $\dist(x,z)\geq \diam (A) /2$ then
\begin{align*}
\Delta(f(A),f(B)) &= \frac{\dist(f(A),f(B))}{\diam f(A)} 
\leq \frac{\dist(f(x),f(y))}{\dist(f(x),f(z))}\\
&\leq \eta\left(\frac{\dist(x,y)}{\dist(x,z)}\right)
\leq \eta\left(2\frac{\dist(A,B)}{\diam A}\right)\\
&\leq \eta\left(2\frac{\dist(A,B)}{\min\{\diam A,\diam B \} } \right) = \eta(2\D(A,B)),
\end{align*}
which proves the right hand side of (\ref{ineq:rel-dist-distortion}). Applying the latter inequality to  $g=f^{-1}:Y\to X$ and using  part $(2)$  of Lemma \ref{lemma:QSproperties}, we obtain
\begin{align*}
\D(A,B) = \D(g(f(A),g(f(B)) 
\leq \theta(2\Delta(f(A),f(B))) = \frac{1}{\eta^{-1}\left(\frac{1}{2\Delta(f(A),f(B))}\right)}.
\end{align*}
Thus, $({2\D(f(A),f(B))})^{-1} \leq \eta \left( {\D(A,B)^{-1}}\right)$, which proves the left hand side of (\ref{ineq:rel-dist-distortion}).
%
\end{proof}

\subsection{Finitely connected domains bounded by quasicircles}

A \emph{quasicircle} is a quasisymmetric image of the unit circle $\partial \mathbb{D}$. The following well-known result of Tukia and V\"as\"al\"a \cite{Tukia Vaisala} provides a complete characterization of quasicircles.

\begin{proposition}\label{quasicircle}
A simple closed curve $\g\subset X$ is a quasicircle if and only if it is doubling and there is a constant $k\geq 1$ such that for every $x,y\in\g$ we have
\begin{equation}\label{quasiconstant}
\min\{\diam(\gamma_1),\diam(\g_2)\} \leq k \cdot d_X(x,y),
\end{equation}
where $\g_1$ and $\g_2$ are the two subarcs  of $\gamma$ with endpoints $x$ and $y$.
\end{proposition}

Here, a metric space $(X, d)$ is \emph{doubling}, if there exists $N \in \mathbb{N}$ such that every ball of radius $r >0$ in $X$ can be covered by at most $N$ balls of radius $r/2$ in $X$.

A quasicircle is a \emph{$k$-quasicircle} for some $k \geq 1$ if it satisfies (\ref{quasiconstant}).  If $\g$ is a $k$-quasicircle and is also doubling with doubling constant $N$, then there exists a quasisymmetry $f$ mapping $\partial \mathbb{D}$ to $\g$, where the distortion function of $f$ depends only on $k$ and $N$. On the other hand, if $f:\partial \mathbb{D} \to \g$ is $\eta$-quasisymmetric then $\g$ satisfies (\ref{quasiconstant}) with $k=2\eta(1)$.

A family $\{\g_i : i \in I\}$ of quasicircles in $X$ is said to consist of \emph{uniform quasicircles} if there exists $k \geq 1$ such that $\g_i$ is a $k$-quasicircle for each $i \in I$. 



\subsection{Lengths of curves}
A \emph{curve} in a metric space $X$ is a continuous function $\g:J\to X$ where $J$ is an interval in $\mathbb{R}$, i.e., there are real numbers $a<b$ such that $J$ has one of the following forms $[a,b],(a,b),[a,b)$ or $(a,b]$. We will often denote the image $\g(J)$ simply by $\g$. We say the curve $\g$ is \emph{rectifiable} if it has finite length: $l(\g)<\infty$. If every compact subcurve of $\g$ is rectifiable, we say that $\g$ is \emph{locally rectifiable}.

If $\Gamma$ is a family of curves in $X$ and $f: X \to Y$ is a homeomorphism, we denote by $f(\Gamma) = \{f\circ \gamma: \gamma \in \Gamma\}$.

Let $E,F$ be subsets of $\overline{X}$. We say that a curve $\g$ in $X$ connects $E$ and $F$ if there is a closed interval $[a,b]\subset\mathbb{R}$ and a continuous path $\g:[a,b]\to \overline{X}$ such that $\g(a)\in E$ and $\g(b)\in F$. We will denote by $\Gamma(E,F;X)$ the family of curves $\gamma$ in $X$ connecting $E$ and $F$. 


For a rectifiable curve $\g:J\to X$, the associated \emph{length function},
$s_{\g}:J\to[0,l(\g)]$ is defined by
$s_{\g}(t) = l(\g([0,t)))$. The \emph{arclength parametrization}  of $\g$ is the unique 1-Lipschitz function $\g_s:[0,l(\g)]\to X$ that satisfies the equation
$\g=\g_s\circ s_{\g}.$

Given a Borel function $\rho:X\to[0,\infty]$ we define the $\rho$-length of a rectifiable curve $\g$ as follows
\begin{align}\label{arclength}
l_{\rho}(\g) : = \int_{\g} \rho ds = \int_0^{l(\g)} \rho(\g_s(t))  dt.
\end{align}

For $f:X\to Y$ and $x\in X$ let
$L_f(x) :=  \limsup_{r\to0} \left( {L_f(x,r)}/{r}\right),
$
where $L_f(x,r)$ is the distortion of $f$ at $x$ at scale $r$ defined in Section \ref{Section:QS-definition}.
The following result, see  \cite[Theorem 5.3]{Vaisala},  will be crucial in the proof of quasi-invariance of transboundary modulus below.

\begin{theorem}\label{length-increasing}
	Suppose $D\subset\mathbb{R}^n$ and $f:D\to\mathbb{R}^n$ is a continuous map. If $\g$ is a locally rectifiable curve in $D$ and $f$ is absolutely continuous on every closed subcurve of $\g$, then $f(\g)$ is locally rectifiable, and for every Borel function $\rho:\mathbb{R}^n\to[0,\infty]$ we have
	\begin{align}
	\int_{f(\g)} \rho ds \leq \int_{\g}(\rho\circ f)\cdot L_f \, ds.
	\end{align}
\end{theorem}

\subsection{Classical Modulus}
Let $(X,d,\mu)$ be a metric space equipped with a Borel measure $\mu$, and $\Gamma$ be a family of curves in $X$. A Borel function $\rho:X\to[0,\infty)$ is called \emph{admissible for $\Gamma$}, denoted by $\rho \wedge \Gamma$, if
$\int_\gamma \rho \ ds \geq 1, \ \forall \gamma \in \Gamma,$
where, as in (\ref{arclength}), $ds$ is the arclength measure of $\gamma$.
%
%
For $p\geq1$, the \emph{$p$-modulus} of $\Gamma$ is defined as
\[
\mod_p (\Gamma) = \inf_{\rho\wedge\Gamma} \int_{X} \rho^p d\mu.
\]

When $(X,d,\mu)$ is locally Ahlfors $2$-regular, i.e., if (\ref{def-ineq:Ahlfors}) is satisfied near every $p$ but only for $r>0$ small enough, and in particular for domains in the plane,  we write  $\mod (\G)$ instead of $\mod_2 (\G)$.

%
The following lemma summarizes some of the most important properties of modulus which will be used in this paper. We say $\G_1$ \emph{minorizes} $\G_2$ and write $\Gamma_1 < \Gamma_2$, if every curve $\gamma \in \Gamma_2$ contains a subcurve $\delta\subset\g$ which belongs to $\Gamma_1$.
\begin{lemma}
Suppose $(X,d,\mu)$ is a metric measure spaces, $p\geq 1$ and $\Gamma_i, i=1,2,\ldots$, are curve families in $X$. Then
  \begin{enumerate}
    \item $\mathrm{\textsc{(Monotonicity)}}$ $\mod_p(\G)\leq \mod_p(\G')$,
        if $\G\subset\G',$
    \item $\mathrm{\textsc{(Subadditivity)}}$ $\mod_p(\G) \leq
        \sum_i\mod_p(\G_i)$, if $\G=\bigcup_{i=1}^{\infty}\G_i,$
    \item $\mathrm{\textsc{(Overflowing)}}$ $\mod_p(\Gamma_1) \geq \mod_p(\Gamma_2)$ if $\Gamma_1 < \Gamma_2$.
  \end{enumerate}
\end{lemma}

\section{Transboundary Modulus} \label{Section:Modulus}

In this section we define the transboundary modulus  introduced by Schramm \cite{Schramm}, and further developed and used by Bonk and Merenkov \cite{Bonk,Merenkov-rel-Schottky}. Our definition slightly differs from those in \cite{Schramm,Bonk,Merenkov-rel-Schottky}, and we explain the reasons for this after the definition. We also prove some properties of the transboundary modulus used below.

\subsection{Definition}
%
%

Let $\Omega$ be a domain in the plane $\mathbb{C}$ and let  $\calK= \{K_i\}_{i=1}^m$ be a finite collection of compact connected sets in $\Omega$.

On the domain $\Omega$ we consider the equivalence relation $\sim_{\calK}$, where $x \sim_{\calK} y$  if and only if $x=y$ or $x$ and $y$ belong to $K_i$ for some $i\in\{1,\ldots,m\}$. 

We denote the corresponding quotient space by
\begin{align*}
\Omega_{\mathcal{K}}=\Omega/\sim_{\calK}.
\end{align*}
The space $\Omega_{\calK}$ is equipped with the quotient topology.  Let  $q : \Omega \to \Omega_{\mathcal{K}}$ be the quotient map. 

Let $K=\cup_{i=1}^m K_i$. The elements of $\Omega_{\calK}$ are the points in $\Omega\setminus K$ and the points corresponding to the subsets $K_i$, denoted by $k_i$. Therefore, 
$$k_i=q(K_i)\in \Omega_{\mathcal{K}}.$$  Since $q$ is injective on $\Omega\setminus K$, we will think of $\Omega\setminus K$ as a subset of $\Omega_{\calK}$ and $q$ restricted to $\Omega\setminus K$ as the  identity map. 

We equip  $\Omega_{\mathcal{K}}$ with a measure $\mu_{\calK}$, which is equal to the 2-dimensional Hausdorff measure $\mathcal{H}^2$  on $\Omega\setminus K$ (or area, as per our convention) and to the counting measure on $q(K)=\Omega_{\calK}\setminus (\Omega\setminus K)$, 
\begin{align*}
  \mu_{\calK} = \calH^2\lfloor_{\Omega\setminus K} +\sum_{i=1}^m \delta_{k_i}.
\end{align*}
A \emph{transboundary mass distribution} on $\Omega_{\calK}$ is an $(m+1)$-tuple
\begin{align*}
\varrho=(\rho;\rho_1,\ldots,\rho_m),
\end{align*}
where 
$\rho:\Omega\setminus K\to[0,\infty)$ is a Borel function
and $\rho_i$ is a non-negative weight corresponding to $K_i$.  Thus $\varrho$ can also be thought of as a Borel function $\varrho:\Omega_{\calK}\to [0,\infty)$.

The \emph{mass} of the transboundary mass distribution $\varrho$ is defined by
\begin{align*}
 A(\varrho)=\int \varrho^2 d\mu_{\calK}=\int_{\Omega\setminus K} \rho^2 d\calH^2 + \sum_{i=1}^m \rho_i^2
\end{align*}

Let  $\g:J\to\Omega_{\calK}$ be a curve where $J\subset \mathbb{R}$ is an interval. Since $\Omega_{\calK}\setminus q(K) = \Omega\setminus K$ is open in $\Omega_{\calK}$, the set $\g^{-1}(\Omega\setminus K)\subset \mathbb{R}$ is a relatively open subset of $J$. Therefore, each connected component of $\g^{-1}(\Omega\setminus K)$ is an interval $J'\subset J$. We say that $\g$ is \emph{locally rectifiable in $\Omega_{\calK}$} if $\g|_{J'}:J'\to \Omega\setminus K$ is locally rectifiable for every component $J'\subset\g^{-1} (\Omega\setminus K)$.

Given a locally rectifiable curve $\g$ in $\Omega_{\mathcal{K}}$ and a transboundary mass distribution $\varrho$, the \textit{$\varrho$-length of $\g$ relative $\calK$} is
\begin{align*}
l_{\varrho}(\g):=\int_{{\g} \cap (\Omega\setminus K)} \rho \ ds + \sum_{i \,: \,k_i\in \g} \rho_i.
\end{align*}

If $\G$ is a family of curves in $\Omega_{\calK}$ we say that a mass distribution $\varrho$ is \emph{admissible for $\G$ relative $\calK$}, and write $\varrho\wedge_{\calK} \G$, if $l_{\varrho}(\g)\geq 1$ for all $\g\in\G$. 

Let $\G$ be a family of curves in $\Omega_{\calK}$. The \emph{{transboundary  modulus of $\G$}} is
\begin{align}\label{def:trmod}
  \Mod_{\Omega,\calK}(\G) =\inf_{\varrho\wedge_{\calK} \G} A(\varrho)= \inf_{\varrho \wedge_{\calK} \G} \int \varrho^2 d\mu_{\calK}.
\end{align}

If $\G$ is a curve family  in $\Omega$  then  we let
\begin{align}\label{trmod-quotients}
\Mod_{\Omega,\calK}(\G) := \Mod_{\Omega,\calK} (q(\G)).
\end{align}

Our definition of transboundary modulus is slightly more general than those in \cite{Bonk,Merenkov}, since we work in the quotient space $\Omega_{\calK}$ like in \cite{Schramm}.  One reason for this is that unlike \cite{Bonk,Merenkov} the mappings we consider cannot be extended to $\Omega$, even continuously (think of the conformal  map $\phi:\mathbb{C} \setminus [-i,i]\to\mathbb{C} \setminus\mathbb{D}$). Nevertheless, using the notation above, for curve families in $\Omega$ our definition coincides with the definitions of Bonk and Merenkov.
Also, we do not use the ends compactification notation used in \cite{Schramm}, since it is more convenient for our applications (see, e.g.,  Lemma \ref{lemma:transmod=0}) to use the notation similar to \cite{Bonk,Merenkov} where the domain $\Omega$ stays fixed, while the families of continua $\calK_{n}$ change with $n$. 

Note that with our convention $\G$ may denote either a family of curves in $\Omega$ or in $\Omega_{\mathcal{K}}$ since transboundary modulus is defined in both cases.

\subsection{Properties of the transboundary modulus}

Some of the properties of transboundary modulus can be proved exactly the same way as for the regular modulus of curve families. However the property of overflowing can be somewhat strengthened. Indeed, we say that $\G_1$ \emph{minorizes} $\G_2$ \emph{relative} $\calK$, and write $\G_1<_{\calK}\G_2$, if for every $\g\in \G_2$ there is a curve $\delta\in\G_1$ such that for the images of the curves $\delta$ and $\g$ under the quotient map $q:\Omega\to \Omega_{\mathcal{K}}$ we have $q(\delta)\subset q(\g)\subset \Omega_{\calK}$.

\begin{proposition}\label{trmodproperties}
Let $\Omega\subset\mathbb{C}$ be a domain, and $\calK=\{K_i\}_{i=1}^m$ be a finite collection of pairwise disjoint compact connected subsets of $\Omega$. Then the following properties are satisfied:
	\begin{enumerate}
	\item $\mathrm{\textsc{(Monotonicity)}}$ $\Mod_{\Omega,{\calK}}(\Gamma) \leq \Mod_{\Omega,{\calK}}(\Gamma')$, if $\Gamma \subset \Gamma'$,
	\item $\mathrm{\textsc{(Subadditivity)}}$ $\Mod_{\Omega,{\calK}}(\Gamma) \leq \sum_{j=1}^{\infty} \Mod_{\Omega,{\calK}}(\Gamma_j)$, if $\Gamma=\bigcup_{j=1}^\infty \Gamma_j$,
	\item $\mathrm{\textsc{(Overflowing)}}$ $\Mod_{\Omega,{\calK}}(\Gamma_1) \geq \Mod_{\Omega,{\calK}}(\Gamma_2)$, if $\Gamma_1 <_{\calK} \Gamma_2$.
	\end{enumerate}
\end{proposition}

\begin{proof}	
To 	prove the properties of overflowing (and therefore of monotonicity) note that if $\Gamma_1 <_{\calK} \Gamma_2$, then any mass distribution $(\rho, \{\rho_i\})$ admissible for $\Gamma_1$ is also admissible for $\Gamma_2$. So $\Mod_{\Omega,{\calK}}(\Gamma_1) \geq \Mod_{\Omega,{\calK}}(\Gamma_2)$.
	
To prove subadditivity assume without loss of generality that $\sum_j \Mod_{\Omega,{\calK}}(\Gamma_j) < \infty$. Fix $\epsilon > 0$.  Then for every $j\geq1$ there is a mass distribution  $(\rho_j, \{\rho_{i,j}\}_{i=1}^m)\wedge_{{\calK}_j}\G_j$ so that
 \[
	 A(\rho_j, \{\rho_{i,j}\}) < \Mod_{\Omega,{\calK}}(\Gamma_j) + \frac{\epsilon}{2^j}.
 \]
	 	
	 	 Let $\tilde\rho = (\sum_j \rho_j^2)^{\frac{1}{2}}$ and $\tilde{\rho}_i = (\sum_j {\rho_{i,j}}^2)^{\frac{1}{2}}$ for $1\leq i \leq m$. Then $(\tilde{\rho},\{\tilde{\rho}_i\})$ is admissible for $\Gamma$ since $\rho\geq\rho_j$, and $\tilde{\rho}_i\geq \rho_{i,j}$ for every $i\in\{1,\ldots,n\}$ and every $j\geq 1$. Therefore,
\begin{align*}
  \Mod_{\Omega,{\calK}}(\Gamma) \leq  A(\tilde{\rho},\{\tilde{\rho}_i\}) < \sum_{j=1}^{\infty} \Mod_{\Omega,{\calK}}(\Gamma_j) + \epsilon.
\end{align*}
Letting $\epsilon \to 0$ finishes the proof.
\end{proof}

%


One of the most important properties of transboundary modulus is that it is a conformal invariant, cf. \cite{Bonk} \cite{Schramm}. Next we show that transboundary modulus is distorted by at most a multiplicative constant under a quasiconformal map. This fact is crucial in the proof of Theorem \ref{thm:main}.

\begin{theorem}[Quasiconformal quasi-invariance of transboundary modulus]\label{quasicomparable}
Suppose $\Omega$ and $\Omega'$ are planar domains and $\calK = \{K_i\}_{i=1}^m$ and  $\calK' = \{K_i'\}_{i=1}^{m}$ are finite collections of compact connected subsets of $\Omega$ and $\Omega'$, respectively. Let $f:(\Omega)_{\calK}\to (\Omega')_{\calK'}$ be a homeomorphism s.t.  
\begin{itemize}
\item  $f(k_i) = k_i', i=1,\ldots m$. 
\item $f|_{\Omega} : \Omega \to \Omega' $
is an $H$-quasiconformal mapping, $H\geq 1$.
\end{itemize}
Then for every curve family $\G\subset \Omega_{\calK}$ we have
\begin{align}\label{ineq:tr-mod-QC}
  {H^{-1}}\Mod_{\Omega,\calK}(\Gamma) \leq \Mod_{\Omega',\calK'}(f(\Gamma)) \leq H \Mod_{\Omega,\calK}(\Gamma),
\end{align}
where $f(\G) = \{f\circ\g : \g\in \G\}$.
\end{theorem}

\begin{proof} Since the inverse of an $H$-quasiconformal map between planar domains is $H$-quasiconformal, it is enough to show only the left inequality in (\ref{ineq:tr-mod-QC}).
We first note that we may assume that 
%
%
for every $\g\in\G$ the mapping $f$ is absolutely continuous on every closed subcurve of $\g\cap \Omega\setminus K$, where $K=\cup_{i=1}^m K_i$. For this, let $\G$ be an arbitrary curve family in $\Omega_{\calK}$ and let
$$\Gamma_1 = \{\gamma \in \Gamma: f \ \textrm{is absolutely continuous on every closed subcurve of} \ \gamma \cap \Omega \backslash K\}.$$ For every $\gamma \in \Gamma \backslash \Gamma_1$ there exists a closed (connected) subcurve $\gamma_{f} \subset \Omega_{\mathcal{K}}$ so that $f$ is not absolutely continuous on it. Let $\Gamma_0 = \{\gamma_f \subset \Omega_{\mathcal{K}}: \gamma \in \Gamma\}$, then $\Gamma_0 <_{\calK} \Gamma \backslash \Gamma_1$. Since $f$ is quasiconformal we obtain that $\mod(\Gamma_0) =0$ cf.,  \cite[page 95]{Vaisala}.

By Proposition \ref{trmodproperties} we have
$\Mod_{\Omega,\calK}(\Gamma \backslash \Gamma_1) \leq \Mod_{\Omega,\calK}(\Gamma_0) = \mod(\Gamma_0)
$
and therefore $\Mod_{\Omega,\calK} (\G\setminus \G_1)=0$.  By subadditivity of transboundary modulus we have
$
 \Mod_{\Omega,\calK} (\G)=\Mod_{\Omega,\calK} (\G_1).
$
In particular, since $\Mod_{\Omega',\calK'} (f(\G_1))\leq \Mod_{\Omega',\calK'} (f(\G))$, in order to obtain the right inequality in (\ref{ineq:tr-mod-QC}) it is enough to show it for $\G=\G_1$. Thus, below we assume that $f$ is absolutely continuous on every closed subcurve of $\g\cap \Omega\setminus K$ for each $\g\in\G$.

Suppose $\varrho'=(\rho';\{\rho_i'\})$ is a mass distribution on $\Omega'_{\calK'}$ admissible for $f(\G)$. Define  $\varrho=(\rho;\{\rho_i\})$ on $\Omega_{\calK}$ as follows,
\begin{align*}
	\rho(x) &= \rho'\big(f(x)\big)\cdot L_f(x), \mbox{ for } x\in \Omega, \\
	\rho_i &= \rho'_i,  1\leq i \leq m.
\end{align*}

Since $f$ is absolutely continuous on every subcurve of $\g\cap \Omega\setminus K$ we have
$L_{\varrho}(\gamma) \geq L_{\varrho'}(f\circ\gamma)\geq 1$,
see Theorem \ref{length-increasing}. Thus $\varrho\wedge_{\calK} \G$ and we have		
\begin{align*}
\Mod_{\Omega,\calK}(\Gamma)
    &\leq \int_{\Omega} \rho^2 d\calH^2 + \sum_{i=1}^m \rho_i^2
    = \int_{\Omega} ( \rho'\circ f)^2  L_f^2 \,d\calH^2 + \sum_{i=1}^m (\rho_i')^n\\
		&\leq H\left(\int_{\Omega}( \rho'\circ f)^2 |J_f| \, d\calH^2 + \sum_{i=1}^m (\rho_i')^2\right)\\
        &\leq H\left(\int_{\Omega'} (\rho')^2 \, d\calH^2 + \sum_{i=1}^m (\rho_i')^2\right).
	\end{align*}
The second to last inequality above holds because a quasiconformal map is differentiable almost every point and for such a point $x\in\Omega$ we have $|Df(x)|^2 = L_f(x)^2 \leq H |J_f(x)|$.	
Taking infimum over all $\varrho$ admissible for $f(\G)$ we obtain the left inequality in (\ref{ineq:tr-mod-QC}).
\end{proof}

\section{Transboundary Lowener Property}\label{Section:TLP}
In this section we define the Transboundary Loewner Property for finitely connected domains (equipped with the inner path metric) and their inverse limits. We also show that QS planar finitely connected domain are TLP if their boundary curves are uniform quasicircles, see Theorem \ref{thm:finitely-connected-embedding}. We start by defining the class of spaces to which we apply the definitions below.

\subsection{Limits of planar domains}\label{Section:limits-of-domains}
The carpets considered in this paper are limits of topologically planar sets $X_n$ equipped with the inner-path metric.  
Recall, that if  $W\subset \mathbb{C}$ is a domain then the inner-path metric $\delta_{W}$ is defined by 
\begin{align}\label{def:path-metric}
\delta_W(x,y)=\inf_{x\overset{\g}{\sim}y}\ell(\gamma), 
\end{align}
where infimum is over all the rectifiable paths $\g$ in $W$ connecting $x$ and $y$, and $\ell(\g)$ denotes the length of $\g$.

Let $\Omega\subset 
\mathbb{C}$ be a simply connected domain, $\calK_1\subset\calK_2\ldots$  a sequence of finite collections of continua (that is, compact connected subsets) of $\Omega$, and $K_n:=\cup_{E\in\calK_n} E$. For  $n\geq1$ we denote by $\overline{W}_n$ the completion of the finitely connected domain $W_n=\Omega\setminus K_n$ in its inner path metric $\delta_{W_n}$. For every $m\leq n$ there is a natural map $\pi_{m,n}:\overline{W}_n \to \overline{W}_m$ induced by the inclusion $W_n\hookrightarrow W_m$.  

Then the carpets $X$ we study are obtained as the inverse limits of sequences $\overline{W}_n$ as above. Specifically, we let 
\begin{align*}
X = \{(x_i)_{i=1}^{\infty} : \pi_{i,i+1}(x_{i+1}) = x_i \}, 
\end{align*}
and equip $X$ with the limiting inner-path metric, i.e., for $x,y\in X$ we define
\begin{align*}
\delta_X(x,y) = \lim_{n\to \infty} \delta_{W_n}(x,y),
\end{align*}
provided the limit exist.

We say that metric carpet $X$ is a \textit{planar inverse limit carpet} if it can be constructed as above. See Section \ref{Section:dyadic-slit-carpets} for the details in the case of slit domains and slit carpets. 

%


\subsection{Transboundary Loewner Property} \label{Secion:TLP-intro}

Recall, that the \emph{relative distance} between compact connected subsets  $E$ and $F$ of a metric space $X=(X,d)$ is defined as follows:
\begin{align}\label{def:rel-dist}
\D(E,F)=\D_X(E,F)= \frac{\textrm{dist}_X(E,F)}{\min\{\textrm{diam}_X(E), \textrm{diam}_X(F)\}}.
\end{align}   

If $X=(W,\delta_W)$ where $W$ is a planar domain and $\delta_W$ is the path metric we will  denote the relative distance by 
$$\D_W(E,F)=\D_X(E,F).$$

We define the transboundary analogue of Loewner condition for finitely connected planar regions equipped with the inner path-metric metric.

In this section we suppose that $\Omega$ is a bounded simply connected domain in the complex plane $\mathbb{C}$,  $\calK$ is a finite  collection of disjoint non-degenerate continua in $\Omega$, and $K:=\cup_{k\in\calK} k$. We will denote by $\mod_{\Omega,\calK}(G)$ the transboundary modulus of $\G$ with respect to $\calK$ (see Section \ref{Section:Modulus} for the details).

\begin{definition}[\textbf{Transboundary Loewner Property}]\label{def:TLP-domains}
Given a function $\Psi:(0,\infty)\to(0,\infty)$, 
and $W=\Omega\setminus K$ as above, we say that $(W,\delta_W)$ (or $\overline{W}$) satisfies the \textsc{$\Psi$-transboundary Loewner property} (or is $\Psi$-\textsc{TLP}) if for every pair of disjoint compact connected sets $E,F\subset W$ we have
 \begin{align}\label{def:TLP}
\Mod_{\Omega,\calK}(\G(E,F)) \geq \Psi(\Delta_W(E,F)),
 \end{align}
 where $\G(E,F)$ is the family of curves in $\Omega$ connecting $E$ and $F$. 
 \end{definition}

We say a sequence $\{W_n\}$ of finitely connected domains is TLP if $W_n$ is $\Psi$-TLP for every $n\geq 1$  for a fixed function $\Psi$.

\begin{definition}\label{def:TLP-inverse-limits}
Suppose $X$  is an inverse limit of a sequence 
 $\{\overline{W}_n\}$, where $W_n$ is a finitely connected planar domain. We say that $X$ satisfies the \textsc{($\Psi$-) transboundary Loewner property} (or is \textsc{($\Psi$-)TLP}) if  $\{W_n\}$ is a \textsc{($\Psi$-) TLP} sequence. 
\end{definition}

Notions similar to the Transboundary Loewner property  have already appeared in the context of quasisymmetric uniformization problems of planar carpets, cf. \cite{Bonk,Merenkov-rel-Schottky}. We apply the TLP in the context of quasisymmetric embeddings of metric carpets which are limits of spaces $\overline{W}_n$ as above, as the number of components  approaches infinity. Unlike the previous works however, the complementary components of $W_n$ in our case may have no interior (e.g. they can be $1$ dimensional linear segments). This allows us to study metric carpets which may not be planar or may not admit bi-Lipschitz embeddings into any finite dimensional Euclidean space. 




\begin{theorem}\label{thm:finitely-connected-embedding}
Suppose $W=\Omega\setminus K$ is a finitely connected domain as above.  If there is a quasisymmetric embedding  $f:(W,\delta_W) \to \mathbb{C}$ such that the boundary components of $f(W)$ are uniform quasicircles then $(W,d_W)$ is TLP.  

In particular, if all the boundary components of $\overline{W}=\overline{(W,\delta_W)}$ are uniform quasicircles and there is a quasisymmetric embedding of $\overline{W}$ into the plane then $(W,\delta_W)$ is TLP. 
\end{theorem}

Theorem \ref{thm:finitely-connected-embedding} is proved in Subsection \ref{Section:QS&TLP}, after we establish the necessary estimates for the transboundary modulus in the plane. 

\begin{remark}
\rm{ Theorem \ref{thm:finitely-connected-embedding} is quantitative in the sense that 
 if $f$ is $\eta$-quasisymmetric and all the boundary components of $W$ are $k$-quasicircles and then  $(W,d_W)$ is $\Psi$-TLP, where $\Psi$ depends only on $\eta$ and $k$. In particular, $\Psi$ does not depend on the number of the boundary components of $W$ (i.e. the number of elements of $\calK$) or their relative distance. This allows us to obtain necessary conditions for quasisymmetrically embedding  certain metric carpets which are limits of finitely connected planar  domains as the number of boundary components approaches infinity. }
\end{remark}

\begin{remark} 
\rm{A crucial point of Theorem \ref{thm:finitely-connected-embedding} is that the metric we consider on the domain $W$ is the path metric $\delta_{W}$ and not the Euclidean metric induced on $W$ from the plane. Therefore, a boundary component of $\overline{(W,\delta_W)}$ can be a topological circle even if the corresponding element of $\calK$ is not a topological disk (e.g., a slit or any other continuum in the plane without interior). Moreover, as a metric space a boundary curve of $\overline{W}$ may be quite exotic (it may not be isometric to a subset of  any finite dimensional Euclidean space, or can have arbitrarily large Hausdorff dimension).}
\end{remark}


\subsection{Transboundary modulus in finitely connected domains}
In general, transboundary modulus cannot be bounded in terms of the classical modulus. The next result however shows that if $\Mod_{\Omega,\mathcal{K}}(\G)$ is small enough and $K_i$'s are uniform quasidisks then transboundary modulus may be bounded from below by the classical modulus. Here we say that a Jordan domain $K\subset \mathbb{C}$ is a \emph{$k$-quasidisk} if  $\partial K$ satisfies (\ref{quasiconstant}).

\begin{lemma}\label{lemma:mod-comparison}
Suppose $\Omega\subset \mathbb{C}$  is a domain and $\mathcal{K} = \{K_i\}_{i=1}^m$ is a collection of (closed) $k$-quasidisks in $\Omega$. If $\G$ is a family of curves in $\Omega$ then 
\begin{align}\label{ineq:TLP1}
\Mod_{\Omega,\mathcal{K}}(\G) \geq \min \{ c_1, c_2  \mod(\G)\}
\end{align}
where the constants $c_1$ and $c_2$ depend only on $k$.
\end{lemma}

For the proof of Lemma \ref{lemma:mod-comparison} we need the following auxiliary results. The first is the well-known Bojarski's Lemma,  see, e.g., \cite[Lemma 4.2]{Bojarski} or \cite[Lemma 5.1]{Merenkov-rel-Schottky}.

\begin{lemma}[Bojarski's Lemma]
Let $B_1,\ldots,B_m$ be any collection of open balls in the plane, $a_1,\ldots,a_m$ be non-negative numbers, and $\lambda\geq 1$. Then there is a constant $C_{\lambda}\geq 1$ which depends only on $\lambda$ such that
\begin{align*}
\int_{\mathbb{C}} \left( \sum_{i=1}^m a_i \chi_{\lambda B_i} \right)^2 dx dy \leq C_{\lambda} \int_{\mathbb{C}}  \left(\sum_{i=1}^m a_i \chi_{B_i} \right)^2 dxdy.
\end{align*}
\end{lemma}

The second lemma gives an upper bound for the number of ``sufficiently large" quasidisks intersecting a given set. Similar results appeared previously for disks, cf. \cite[Lemma 5.2]{Merenkov-rel-Schottky}, or for uniformly relatively separated quasi-round sets \cite[Lemma 8.2]{Bonk}.  We will need a version which works for quasidisks which are not necessarily uniformly separated. To state the next result we need the notion of fat sets due to Schramm \cite{Schramm}. 

Let $\tau>0$. A set $K\subset \mathbb{C}$ is said to be \emph{$\tau$-fat} if for every $x\in K$ and $r>0$ such that $B(x,r)$ does not contain $K$ we have
\begin{align}\label{fat-sets}
\calH^2(B(x,r)\cap K) \geq \tau \calH^2(B(x,r)). 
\end{align}

\begin{lemma}\label{lemma:cardinality}
Suppose $E$ is a planar continuum, and $\tau>0$. Let  $\{K_i\}_{i\in I}$ be a collection of disjoint $\tau$-fat sets in the plane, such that 
\begin{align}\label{disjoint-fat-sets}
K_i\cap E \neq \emptyset \quad \mathrm{and} \quad  \lambda\diam K_i \geq \diam E,
\end{align}
for some $\lambda\geq 1$. 
Then $\mathrm{card}(I)\leq N$, where $N=N(\tau,\lambda)$ depends only on $\tau$ and $\lambda$.
\end{lemma}

\begin{proof}
Without loss of generality assume $0\in E$. 
Denote $d=\diam E$. Since $K_i$ intersects $E$ we have that $K_i$ intersects the ball $B=B(0,d)$. 

Denote $\delta=\lambda^{-1}$ and let $I_1$ be the collection of $i\in I$ such that  $K_i$ intersects the circle $\{|z| = (1+\delta) d\}$.

For $i\in I_1$ pick $x_i\in K_i\cap \{|z|=d\}$. Since $K_i\cap \{|z| =(1+\delta)d\} \neq \emptyset$ we have that $B(x_i,\delta d)$ does not contain $K_i$ and therefore
\begin{align*}
\calH^2(B(x_i,\delta d)\cap K_i) \geq \tau \calH^2(B(x_i,\delta d)).
\end{align*}
Since $K_i$'s are disjoint we have
\begin{align*}
\mathrm{card} I_1  \leq \frac{\calH^2(B(0,(1+\delta)d)\setminus B(0,(1-\delta) d))}{\tau \calH ^2(B(x_i,\delta d))}=\frac{1}{\tau} \frac{(1+\delta)^2-(1-\delta)^2}{\delta^2}=\frac{4\lambda}{\tau}.
\end{align*}

To estimate $\mathrm{card}(I\setminus I_1)$, observe that since $ \diam(K_i) \geq  d/\lambda$, there is a point $y_i\in K_i$ such that $B(y_i,d/\lambda)$ does not contain $K_i$. Hence, $\calH^2(B(y_i,d/\lambda) \cap K_i) \geq \tau\calH^2(B(y_i,d/\lambda))$. Since $K_i\subset B(0,(1+\delta) d)$ for $i\in I \setminus I_1$ it follows that
\begin{align*}
\mathrm{card} (I\setminus I_1) \leq \frac{\calH^2(B(0,(1+\delta)d))}{\tau \calH^2(B(y_i,d/\lambda))} = \frac{(1+\lambda)^2}{\tau} .
\end{align*}
Thus $\mathrm{card}(I)\leq N(\tau,\lambda)=\tau^{-1}(4\lambda+(1+\lambda)^2) \leq (3+\lambda)^2/\tau$.
\end{proof}

By \cite[Corollary 2.3]{Schramm} every $k$-quasidisk  is $\tau$-fat  with $\tau$ depending only on $k$. Therefore, by Lemma \ref{lemma:cardinality} we obtain the following. 

\begin{corollary}\label{cardinality-quasidisks}
Suppose $E$ is a planar continuum. Let $k\geq1$ and   $\{K_i\}_{i\in I}$ be a collection of disjoint $k$-quasidisks in the plane, such that 
(\ref{disjoint-fat-sets}) holds. Then $\mathrm{card}(I)\leq N_1$, where $N_1=N_1(k,\lambda)$ depends only on $k$ and $\lambda$.
\end{corollary}

\begin{proof} [Proof of Lemma \ref{lemma:mod-comparison}]

Since every $K_i$ is a $k$-quasidisk, it follows (see e.g., \cite[Proposition 4.3]{Bonk}) that there is a $\lambda\geq 1$, which depends only on $k$, such that $K_i$'s are \emph{$\lambda$-quasi-round}, i.e., for every $i=1,\ldots,m,$  there is a ball $B_i=B(x_i,R_i) \subset \mathbb{C}$ such that $\lambda^{-1}B_i\subset K_i \subset B_i.$

Let $C = N_1(k,2\lambda)$, where $N_1$ is the constant from Corollary \ref{cardinality-quasidisks}. Assume that $\Mod_{\Omega,\mathcal{K}} (\G) \leq 1/(8 C^2)$. Let $\epsilon<1/(8C^2)$, and
choose a mass distribution $\varrho=(\rho;\rho_1,\ldots,\rho_n)$ that is admissible for $\G$ relative $\mathcal{K}$ and such that 
\begin{align*}
A(\varrho) \leq \Mod_{\Omega,\mathcal{K}} (\G) +\epsilon \leq \frac{1}{4 C^2}.
\end{align*}

 Define $g:\Omega\to[0,\infty)$ as follows
\begin{align*}
g= 2\left(\rho + \sum_{i=1}^m \frac{\rho_i}{R_i} \chi_{ 2B_i\cap \Omega}\right).
\end{align*}
To show that $g$ is admissible for $\G$, pick a curve $\g\in\G$ and let 

\begin{align*}
I_{\g} = \{i\in I : \g \cap K_i \neq \emptyset \mbox{ and } 2 \lambda\diam(K_i) \geq \diam (\g)\}.
\end{align*}

Then
\begin{align*}
\int_{\g} g ds 
&= 2 \left(\int_{\g\cap \Omega} \rho ds +  \sum_{i: \, \g\cap 2 B_i \neq \emptyset} \frac{\rho_i}{R_i} \int_{\g\cap 2 B_i}  ds \right)  \nonumber \\
&\geq2 \left( \int_{\g\cap \Omega} \rho ds + \sum_{i\in I\setminus I_{\g}: \, \g\cap K_i \neq \emptyset} \frac{\rho_i}{R_i} \int_{\g\cap 2 B_i}  ds \right).
\end{align*}
For $i\in I\setminus I_{\gamma}$ we have 
$\diam (\g) > 2\lambda \diam (K_i)\geq 2 \lambda \diam (\lambda^{-1} B_i) = 2\diam B_i.$
Hence, $\g$ is not contained in $2B_i$ and if, additionally, $\gamma \cap K_i \neq \emptyset$ we have that $\int_{\g\cap 2 B_i} ds \geq R_i$.
Thus, 
\begin{align}\label{est:length}
\int_{\g} g ds \geq 2 \left( \int_{\g\cap \Omega} \rho ds + \sum_{i\in I\setminus I_{\g}: \, \g\cap K_i \neq \emptyset} {\rho_i}  \right).
\end{align}
To estimate the right hand side from below observe that $\rho_i \leq \sqrt{A(\rho)}\leq 1/(2C)$. Moreover,  by Corollary \ref{cardinality-quasidisks} we have $\mathrm{card}(I_{\g})\leq C$. Thus $\sum_{i\in I_{\g}} \rho_i \leq \frac{1}{2}$
and
\begin{align*}
\int _{\g\cap\Omega} \rho ds + \sum_{i\in I\setminus I_{\g} \, : \, \g \cap K_i\neq \emptyset} \rho_{i} \geq l_{\varrho} (\g) -\frac{1}{2 }\geq \frac{1}{2},
\end{align*}
since $\varrho$ is admissible for $\G$ relative $\mathcal{K}$. Hence by (\ref{est:length}) $\int_{\g} gds \geq 1$ and $g$ is admissible for $\G$.

Thus we can estimate the modulus as follows, 
\begin{align} \label{est:trans-modulus-sum}
\mod (\G) 
&\leq \int_{\Omega} g^2 \, dx dy\leq
8 \left( \int_{\Omega}\rho^2 dxdy + \int_{\Omega} \left(\sum_{i=1}^m \frac{\rho_i}{R_i} \chi_{2B_i\cap \Omega}\right)^2 dxdy \right).
\end{align}
By Bojarski's Lemma and because the balls $\{\lambda^{-1} B_i\}_{i=1}^m$ are pairwise disjoint, we have
\begin{align*}
\int_{\Omega} \left(\sum_{i=1}^m \frac{\rho_i}{R_i} \chi_{2B_i\cap \Omega}\right)^2 dxdy
&\leq \int_{\mathbb{C}} \left(\sum_{i=1}^m \frac{\rho_i}{R_i} \chi_{2 B_i}\right)^2 dxdy \\
&\leq C_{2\lambda} \int_{\mathbb{C}} 
\left(
\sum_{i=1}^m   \frac{\rho_i}{R_i} \chi_{\lambda^{-1}B_i} \right)^2 
dx dy\\
& = C_{2\lambda} 
\sum_{i=1}^m  \left( \frac{\rho_i}{R_i} \right)^2  \calH^2(\lambda^{-1} B_i)
\\
&= \frac{\pi C_{2\lambda}}{\lambda^2} \sum_{i=1}^m   {\rho_i}^2 .
\end{align*}
Therefore, by (\ref{est:trans-modulus-sum}) we obtain
\begin{align*}
\mod (\G) \leq 
8 \max\left(1, \frac{\pi C_{2\la}}{\la^2}\right) \Mod_{\Omega,\mathcal{K}} (\G). \qquad
\qedhere
\end{align*}

\end{proof}

\subsection{QS embeddings and TLP}\label{Section:QS&TLP} Here we combine the results above to prove Theorem \ref{thm:finitely-connected-embedding}.

\begin{proof}[Proof of Theorem \ref{thm:finitely-connected-embedding}] 
Let $E,F\subset W$ be disjoint nontrivial continua and let $\G=\G(E,F;\Omega)$. Abusing the notation slightly we will also denote by $\G$ the family of curves $q(\G)$ in $\Omega_{\calK}$, where $q$ is the quotient maps $q:\Omega\to\Omega_{\calK}$.  Then there is a domain $\Omega'\subset \mathbb{C}$ and continua $\calK'$ in $\Omega'$ so that $f(W)=W'=\Omega'\setminus K'$. Let $\G'=\G(f(E),f(F);\Omega')$. Then, using the definition of transboundary modulus, the observation that $f(q(\G)) = q'(\G')$ and quasi-invariance of transboundary modulus we obtain 
\begin{align*}
\Mod_{\Omega,\mathcal{K}}(\G) &=\Mod_{\Omega,\mathcal{K}}(q(\G))\\
&\asymp \Mod_{\Omega',\calK'} (f(q(\G)))\\
&= \Mod_{\Omega',\calK'} (q'(\G'))\\
&=\Mod_{\Omega',\calK'} (\G')\\
&\geq \min \{ c_1, c_2  \mod(\G')\}
\end{align*}
where the last inequality follows from Lemma \ref{lemma:mod-comparison}. Since $\Omega'$ is a quasidisk it is $\phi$-Loewner  and hence
\begin{align*}
\mod(\G') \geq \phi(\Delta(f(E),f(F))) \geq \phi \left(\frac{1}{2\eta(\Delta_W(E,F)^{-1})}\right).
\end{align*}
Combining the estimates above we obtain 
$\mod_{\Omega,\calK} \G(E,F) \geq \Psi(\Delta_W(E,F)),$
where $\Psi(t) = \min\left\{c_1,c_2 \phi\left(\frac{1}{2\eta(1/t)}\right)\right\}$. This proves the first part of the theorem.

Note that if the boundary components of $\overline{W}$ are (non-planar) uniform quasicircles and $f:W\to \mathbb{C}$ is a quasisymmetry then the boundary components of $f(W)$ are uniform (planar) quasicircles. Hence, the second part of the theorem follows from the first part.
\end{proof}

\section{Dyadic Slit Carpets}\label{Section:slit-carpets}

\subsection{Metric carpets}
The classical Sierpi\'nski carpet $\mathbb{S}_{1/3}$ is the subset of the plane obtained as follows: Divide the unit square $[0,1]^2$ into nine congruent squares of side-length $1/3$ with disjoint interiors, and let $E_1$ be the closed set obtained by removing the interior of the middle square from $[0,1]^2$. Assume that for $i\geq 1$ the set $E_i$ has been constructed and is a union of finitely many closed squares with sidelength $1/3^i$ (and disjoint interiors). Dividing each such square in $E_i$ into $9$ subsquares and removing the interiors of middle squares produces the set $E_{i+1}\subset E_i.$ The classical Sierpi\'nski carpet $\mathbb{S}_{1/3}$ is  defined as the compact set $\cap_{i\in I} E_i.$


 The following  theorem of Whyburn \cite{Whyburn} characterizes the subsets of the sphere which are homeomorphic to $\mathbb{S}_{1/3}$.

\begin{theorem}[Whyburn]\label{metricsierpinski}
	Suppose $D_i \subset \mathbb{S}^2, \ i \geq 0$, is a sequence of topological disks satisfying the conditions:
	\begin{enumerate}
		\item $\overline{D_i} \cap \overline{D_j} = \emptyset, \ \textrm{for} \ i \neq j,$
		\item $\diam(D_i) \to 0, \ \textrm{as} \ i \to \infty,$
		\item $\overline{(\bigcup_i D_i)}=\mathbb{S}^2.$
	\end{enumerate}
	
	Then the compact set $\mathbb{S}^2 \backslash \bigcup_i D_i$ is homeomorphic to the standard Sierpi\'nski carpet $\mathbb{S}_{1/3}$.
\end{theorem}

If $X$ is a metric carpet then a topological circle $\g\subset X$ is called a \emph{peripheral circle} if $X\setminus \g$ is connected, i.e., $\g$ is a non-separating curve in $X$. From Whyburn's theorem it follows that $\g\subset X$ is a non-separating curve  if and only if there is a homeomorphism mapping $X$ to $\mathbb{S}_{1/3}$ and $\g$ to the boundary of one of the complementary domains of $\mathbb{S}_{1/3}$ in the plane.

In this section we construct a class of metric carpets called \emph{dyadic slit carpets} which are the main object of study in this paper. Dyadic slit carpets include the slit carpet considered by Merenkov in \cite{Merenkov} and were also considered by the first author in \cite{Hakobyan}.

\subsection{Dyadic slit domains and the inner metric} 

Let $\mathbb{U}=(0,1)\times(0,1)$ in $\mathbb{R}^2$. We say that $\Delta\subset \overline{\mathbb{U}}$ is a \emph{dyadic square of generation $n$} if there exist $i,j \in \{0,1,2,\ldots,2^n-1\}$ such that
\begin{align*}
\Delta=\left[\frac{i}{2^n},\frac{i+1}{2^n}\right] \times \left[\frac{j}{2^n},\frac{j+1}{2^n}\right].
\end{align*}
We will denote by $\mathcal{D}_n$ be the collection of all dyadic squares of generation $n$ and by $\mathcal{D}=\cup_{n=0}^{\infty}\mathcal{D}_n$ the collection of all dyadic squares in $[0,1]^2$. The sidelength of a dyadic square $\Delta$ will be denote by $l(\Delta)$. Thus, $\Delta\in\mathcal{D}_n$ if and only if $l(\Delta)=1/2^n$.


Given a sequence $\mathbf{r}=\{r_n\}_{n=0}^\infty$ such that $r_n\in(0,1)$ for $n=0,1,2,\ldots,$ we next construct the corresponding sequence of ``slit" domains $S_n=S_n(\mathbf{r})$ in $\mathbb{U}$. For every dyadic square $\Delta$ of generation $n$ we denote by $s(\Delta)$ the closed vertical slit in $\Delta$ of length $r_n l(\Delta)$, whose center coincides with the center of $\Delta$. More precisely, if $(x,y)$ is the center of $\Delta\in \mathcal{D}_n$ then
\begin{align*}
  s(\Delta) = \{x\}\times \left[y-\frac{r_n}{2^{n+1}}, y+\frac{r_n}{2^{n+1}}\right].
\end{align*}
We say that a slit $s=s(\Delta)\subset{\D}$ is \emph{a slit of generation $n$} if $\Delta\in\mathcal{D}_n$, for some $n\geq 0$. 
For $n\geq0$ let
\begin{align*}
  \calK_n&=\calK_n(\mathbf{r})=\{s(\Delta): \D\in\mathcal{D}_0\cup\ldots\cup\mathcal{D}_n \} \quad \mbox{and}\\
  K_n &= \bigcup_{s\in\calK_n} s = \bigcup_{i=0}^n \bigcup_{\D\in\calD_i} s(\D)
\end{align*} 
be the collection of all slits of generation at most $n$  and their union, respectively. We will also use the following convention: $K_{-1}=\emptyset$.

Similarly, for  the collection of all slits and their union let
\begin{align*}
\calK &=\calK(\mathbf{r})=\{s(\Delta): \Delta\in\mathcal{D}\} \mbox{ and } \\
K&=\bigcup_{s \in \calK} s=\bigcup_{\D\in\calD}s(\D)
\end{align*}   

Finally, let $S_0=\mathbb{U}$ and for $n\geq1$, let
\begin{align}
  {S}_n = \mathbb{U}\backslash K_{n-1} = \mathbb{U}\setminus \bigcup_{i=0}^{n-1}\bigcup_{\D\in\mathcal{D}_i} s(\D),
\end{align}
where $\mathbb{U}$ is the open unit square as usual. We call $S_n$ the \emph{dyadic slit domain of generation $n\geq 0$}.

To define the metric carpet $\mathscr{S}_{\mathbf{r}}$, we first let $\mathscr{S}_n$ be the completion of the domain ${S}_n$ in its path metric $d_{S_n}$.
Recall that the path metric $d_{\Omega}$ on a domain $\Omega\subset \mathbb{R}^n$ is defined by
\[
d_{\Omega}(x,y):=\inf\{l(\gamma): \g\subset\Omega \mbox{ s.t. } \g(0)=x, \g(1)=y\},
\]
for all $x,y\in\Omega$, where $l(\g)$ denotes the length of a rectifiable curve $\g$ in $\Omega$, and the infimum is over all rectifiable curves in $\Omega$ connecting $x$ and $y$. The metric on $\mathscr{S}_n$ will be denoted by $d_{\mathscr{S}_n}$. Note that $\S_0$ is isometric to $[0,1]^2$ equipped with the Euclidean metric.

A boundary component of $\mathscr{S}_n$ corresponding to a slit of a dyadic square $\D\in \calD_m$ of generation  $m\leq n-1$ will be called a \emph{slit of $\mathscr{S}_n$ of generation $m$}. The slit of generation $0$ in $\S_n$ will be called the \textit{the central slit of $\S_n$}. The boundary component of $\mathscr{S}_n$ corresponding to $\partial([0,1]^2)$ will be called the \emph{outer square of $\S_n$}.

\subsection{Dyadic slit carpets}\label{Section:dyadic-slit-carpets}

For every $m, n \in \mathbb{N} \cup \{0\}$ with $m \leq n$ there exists a natural $1$-Lipschitz projection 
$$\pi_{m,n}: \mathscr{S}_n \to \mathscr{S}_m$$ obtained by identifying the points on the slits of $\mathscr{S}_n$ that correspond to the same point of $\mathscr{S}_m$. More precisely, if $p,q\in\S_n$ then $\pi_{m,n}(p) = \pi_{m,n}(q)$, {whenever} $\ d_{\mathscr{S}_m}(p,q) = 0.$ Note that all the boundary components of $\S_n$ are topological circles. Moreover, every slit of diameter $d>0$ in $\S_n$ is isometric to the square $\partial([0,d/2]\times[0,d/2])\subset\mathbb{R}^2$ equipped with the metric induced from the $\ell^1$ norm on $\mathbb{R}^2$.

As a topological space, the \emph{dyadic slit Sierpi\'nski carpet corresponding to $\mathbf{r}$} is defined as the inverse limit of the system $(\mathscr{S}_n,\pi_{m,n})$, and is denoted by $\mathscr{S}_{\mathbf{r}}$. More explicitly,
\begin{align}
  \mathscr{S}_{\mathbf{r}} = \left\{ (p_0,p_1,\ldots) : p_i \in \mathscr{S}_i \mbox{ and } p_i=\pi_{i,i+1}(p_{i+1}) \right\}.
\end{align}
If the sequence $\mathbf{r}$ is understood from the context, we will denote $\mathscr{S}_{\mathbf{r}}$ simply by $\mathscr{S}$. 

The inverse limits of the slits and outer squares of $\mathscr{S}_n$ are topological circles and will be called the slits and outer square of $\mathscr{S}$, respectively. Clearly, the slits are dense in $\mathscr{S}$, i.e., for every point $p$ in $\mathscr{S}$ and every neighborhood $U$ of $p$, there exists a slit of $\mathscr{S}$ that intersects $U$.

The diameter of each $\mathscr{S}_n$ is clearly bounded by $2$. If $x = (x_0,x_1, \ldots )$ and $
y = (y_0, y_1, \ldots)$ are points in $\mathscr{S}$, we define a distance between them by
\[
d_{\mathscr{S}}(x,y)=\lim_{n\to\infty} d_{\mathscr{S}_n}(x_n,y_n)
\]
Since every $\pi_{m,n}$ is $1$-Lipschitz, $(d_{\mathscr{S}_n}(p_n, q_n))$ is a non-decreasing bounded sequence, and thus $d_{\mathscr{S}}(p, q)$ exists and defines a metric on $\mathscr{S}$.

For each $n \geq 0$, there are natural projection maps
\begin{align*}
\pi_n:\mathscr{S}\to\mathscr{S}_n,\\ 
\pi_{0,n}:\mathscr{S}_n\to\mathscr{S}_0.
\end{align*}
To simplify notations, we will denote the projection $\pi_0$ of $\S$ onto the unit square by $\pi$. Thus, for every $n\geq 1$ we have 
\begin{align*}
\pi = \pi_{0,n} \circ \pi_n : \S\to[0,1]^2.
\end{align*} 
%

It was shown in \cite{Hakobyan} (see also \cite{Merenkov})  that the metric space $\mathscr{S}$ corresponding to a general collections of slits $\{s_i\}_{i=1}^{\infty}\subset (0,1)^2$ is homeomorphic to the Sierpi\'nski carpet $\mathbb{S}_{1/3}$, provided that the slits are uniformly relatively separated, dense in $[0,1]^2$ and $\diam(s_i)\to0$ as $i\to\infty$. In fact, the proof of \cite[Lemma 2.1]{Merenkov} easily generalizes to show that $\S_{\mathbf{r}}$ is homeomorphic to $\mathbb{S}_{1/3}$  even for an arbitrary sequence $\mathbf{r}=\{r_i\}_{i=0}^{\infty}$,  where $0<r_i<1$.

When talking about a {dyadic square of generation $n$ in $\mathscr{S}$}, we mean the subset of $\pi^{-1}(\Delta), \Delta \in \mathcal{D}_n$, which can be thought of as a slit carpet with respect to $\{r_i\}_{i=n}^\infty$ constructed in $\D$ instead of $\mathbb{U}$. More precisely, we say that $T\subset\S$  is a \emph{dyadic square of generation $n$ in $\mathscr{S}$}, if there is a dyadic square $\D\in\calD_n$ such that 
\begin{align}
T=T_{\Delta}=\overline{\pi^{-1}(\mbox{int}(\Delta))} .
\end{align}
We will also denote
\begin{align*}
\partial T_{\D}:= T_{\D} \setminus \pi^{-1} (\mbox{int}(\Delta)).
\end{align*} 
Thus $\partial{T_{\D}}$ is the ``outer square" of $T_{\D}$. For all $m,n\geq 0$ a \emph{dyadic square of generation $m$ in $\mathscr{S}_n$} is the  image of a dyadic square of generation $m$ in $\mathscr{S}$ under $\pi_n$. Note that for $m>n$ dyadic squares of generation $m$ in $\S_n$ do not contain slits in their interiors and therefore are isometric to Euclidean squares. 


Define a projection map $\mathrm{proj}(x,y)=x$ for $\forall \ (x,y) \in [0,1]^2$. A curve $\gamma: [a,b] \to  \mathscr{S}$ in a slit carpet  is called \emph{vertical} if $\mathrm{proj}(\pi(\gamma([a,b]))$ is a point, i.e., the first coordinate of $\pi(\gamma)$ is a constant. A curve which is not vertical is called \emph{nonvertical}.

The following properties are from \cite{Merenkov} and \cite{Hakobyan}. We state them without proof.

\begin{lemma}\label{slitmetric}
 There exists a constant $0 < c < 1$, which does not  depend of $n$, such that  $\forall p \in \mathscr{S}$ and $0 < r < \diam(\mathscr{S})$ there exists a point $q \in \mathscr{S}_n, n \geq 0$ such that
\begin{align}
  B(q, c r) \subset \pi_n(B(p, r)) \subset B(\pi_n(p), r).
\end{align}
\end{lemma}

\begin{lemma}\label{slitmeasure}
There exists a constant $C \geq 1$, independent of $n \geq 1$, such
that for any Borel set $E \subset \mathscr{S}$ we have
\[
\frac{1}{C}\mathcal{H}^2(\pi_n(E)) \leq \mathcal{H}^2(E) \leq C \mathcal{H}^2(\pi_n(E)).
\]
In addition, $\mathscr{S}$ and $\mathscr{S}_n$ are Ahlfors $2$-regular with the same Ahlfors regularity constant and $N$-doubling with the same doubling constant for every $n$.
\end{lemma}

\begin{lemma}
The metric space $\mathscr{S}$ equipped with $\mathcal{H}^2$ is a metric Sierpi\'nski carpet which is doubling and Ahlfors $2$-regular.
\end{lemma}

\section{QS planarity implies weak TLP}\label{Section:necessary-condition} In this section we provide a necessary condition for the existence of a quasisymmetric embedding of the slit carpet $\S_{\mathbf{r}}$ into the plane. 
%
%
%
%
This condition is an estimate on the transboundary modulus relative to the collection of slits $\calK_n$ . Below we use the notations introduced in Section \ref{Section:slit-carpets}. In particular,  $\mathbb{U}=(0,1)^2$. Moreover, 
we let
\begin{align}\label{notation}
\begin{split}
L &=\{(0,y) : 0\leq y \leq 1 \}\subset\partial\mathbb{U},\\
R &=\{(1,y) : 0\leq y \leq 1 \}\subset\partial\mathbb{U},\\
\G_{} &=\G(L,R;\mathbb{U}).
\end{split}
\end{align}
Thus, $\G$ is the family of curves in $\overline{\mathbb{U}}$ connecting the vertical sides of $\mathbb{U}$.


\begin{lemma}\label{lemma:necessary-condition-for-QS-embedding}
Suppose there is an $\eta$-quasisymmetric embedding  $f: \mathscr{S}_{\mathbf{r}}\hookrightarrow\mathbb{R}^2$ of the slit carpet $\S=\S_{\mathbf{r}}$ into the plane. Then there is a constant $c>0$ which depends only on $\eta$ such that for every $n>0$ we have
\begin{align}\label{ineq:trmod>0}
\Mod_{\mathbb{U},\calK_n} (\G) \geq c.
\end{align}
%
\end{lemma}

\begin{remark}
\rm{ Note that by Theorem \ref{thm:slit-carpets-TLP}  the carpet $\S_{\mathbf{r}}$ satisfies the Transboundary Loewner Property if there is a quasisymmetric embedding of $\S_{\mathbf{r}}$ into the plane. This easily implies Lemma \ref{lemma:necessary-condition-for-QS-embedding}, since $\D(L,R) = 1$ and therefore by Theorem \ref{thm:slit-carpets-TLP} 
$$\Mod_{\mathbb{U},\calK_n}\G(L,R) \geq \Psi(1)>0.$$

Thus, condition (\ref{ineq:trmod>0}) may be thought of as a (very) weak form of TLP. However, we show that it is in fact equivalent to TLP.
To see this we will combine Lemma \ref{lemma:necessary-condition-for-QS-embedding} with the bounds in Section \ref{Section:trmod-in-slit-domains} and show that if there is a QS embedding $f:\S_{\mathbf{r}}\rightarrow\mathbb{R}^2$ then $\mathbf{r}\in\ell^2$. Then, with the help of Bonk-Kleiner's theorem, we will show that if $\mathbf{r}\in\ell^2$ then the finite approximations $\S_n$ of $\S$ can be embedded into the plane, via uniformly quasisymmetric maps. 
 Finally, from Theorem \ref{thm:finitely-connected-embedding} it would follow that $\S_i$'s are uniformly TLP. In short, we have the following implications:
$$\S_{\mathbf{r}}\hookrightarrow \mathbb{R}^2 \Rightarrow (\ref{ineq:trmod>0}) \Rightarrow \mathbf{r}\in\ell^2 \Rightarrow \widehat{\S}_{\mathbf{r}} \cong \mathbb{S}^2 \Rightarrow \S_{\mathbf{r}}\mbox{ is TLP }.$$
}
\end{remark}

To prove Lemma \ref{lemma:necessary-condition-for-QS-embedding} we will first show that a quasisymmetric embedding $f:\S\hookrightarrow\mathbb{R}^2$ descends to uniformly quasiconformal mappings $f_n:\S_n\hookrightarrow\mathbb{R}^2$, which are quasisymmetric on the ``outer square", see Lemma \ref{lemma:qc-extension}.


For $n\geq 1$, we will denote by $\Pi_{n}$ and $\tilde{\Pi}_{n}$ the preimages of  the dyadic grid of generation $n$ in $\overline{\mathbb{U}}=[0,1]^2$ under the projections $\pi_{0,n}$ and $\pi$ in $\S_{n}$ and $\S$, respectively. In other words we have 
	\begin{align*}
	\Pi_{n}&=\pi_{0,n}^{-1} \left(\bigcup_{\Delta\in\mathcal{D}_{n}} \partial \Delta \right) \subset \S_{n}, \qquad
	\tilde{\Pi}_{n}=\pi^{-1}\left(\bigcup_{\Delta\in\mathcal{D}_{n}} \partial \Delta \right) \subset \S.
	\end{align*}
From the definitions it follows that $\pi_n|_{\tilde{\Pi}_{n}}$ is a homeomorphism. In fact more is true.

	\begin{lemma}\label{lemma:-bilip-squares}
	For every $n\geq 0$, the mapping $\pi_n|_{\tilde{\Pi}_{n}}$, i.e., the restriction of the projection maps $\pi_n:\S\to\S_n$ to ${\tilde{\Pi}_{n}}$ is bi-Lipschitz. More precisely, if  $p,q\in{\tilde{\Pi}_{n}}$ then
	\begin{align}\label{ineq:boundary-lipschitz}
	d_{\S_n}(\pi_n(p),\pi_n(q)) \leq d_{\S} (p,q) \leq 3 d_{\S_n}(\pi_n(p),\pi_n(q)).
	\end{align}
	\end{lemma}
\begin{proof}
		The left inequality in (\ref{ineq:boundary-lipschitz}) follows from the fact that the sequence
		$d_{\S_n}(\pi_n(p),\pi_n(q))$ is non-decreasing in $n$. 
		
		To obtain the right inequality in (\ref{ineq:boundary-lipschitz}) we recall the following notation from Section \ref{Section:dyadic-slit-carpets}. Suppose $n\geq 0$ and $ \Delta\in\mathcal{D}_{n}$ is a dyadic square. Let $T=T_\Delta$ be the corresponding ``dyadic square" in $\S$, i.e., $T_{\Delta} = \overline{\pi^{-1}(\mathrm{int}(\D))}$
where the closure is in $d_{\S}$ metric.

First, assume that $p,q\in\partial T_{\D}$ for some $\D\in\calD_{n}$.  If $\pi(p)$ and $\pi(q)$ belong to the same edge of the square $\partial \Delta$ then  
$$d_{\S} (p,q) = d_{\S_0}(\pi(p),\pi(q))=|\pi(p)-\pi(q)|.$$ 
On the other hand, if $\pi(p)$ and $\pi(q)$ belong to different edges of the Euclidean square $\partial \Delta$ then there are at most two corner points $z_1,z_2$ of $\partial \Delta$ between $\pi(p)$ and $\pi(q)$ on $\partial{\Delta}$ such that
		\begin{align*}
		|\pi(p)-\pi(z_1)| + |\pi(z_1)-\pi(z_2)|+ |\pi(z_2)-\pi(q)| \leq 3|\pi(p)-\pi(q)|.
		\end{align*}
Therefore,
\begin{align*}
d_{\S} (p,q)
&\leq d_{\S} (p,z_1) + d_{\S} (z_1,z_2)+ d_{\S} (z_2,q) \\
& = |\pi(p)-\pi(z_1)| + |\pi(z_1)-\pi(z_2)|+ |\pi(z_2)-\pi(q)|\\
&\leq 3 d_{\S_0}(\pi(p),\pi(q))\\
& \leq 3 d_{\S_n}(\pi_n(p),\pi_n(q)).
\end{align*}

More generally, suppose $p,q\in\tilde{\Pi}_{n}$. Consider a curve $\g$ connecting  $\g(0)=\pi_n(p)$ and $\g(1)=\pi_n(q)$ in $\S_n$ of minimal length. It is easy to see that such a curve exists, and  it is, in fact, a preimage of a piecewise linear curve in $\overline{\mathbb{U}}$ under $\pi_{0,n}$. Observe that there are points $\zeta_j, j=0,\ldots, k+1,$ on $\g$ such that: $(i)$ $\zeta_0=\pi_n(p)$, $\zeta_{k+1}=\pi_{n}(q)$, $(ii)$ for every $j$ the two consecutive points $\zeta_j$ and $\zeta_{j+1}$  belong to the outer boundary of the same dyadic square $\pi_{n}(T_{\D})\subset\S_{n}$ for some $\D\in\calD_{n}$, and $(iii)$ the following equality holds $d_{\S_n}(\pi_n(p),\pi_n(q)) = \sum_{j=0}^k d_{\S_n}(\zeta_j,\zeta_{j+1}).$ Indeed, this can be achieved by letting $\zeta_1$ be the ``last point of exit" of $\g$ from the (closed) square $T_{\D}$ containing $\zeta_0=\g(0)$, and continuing by induction. 

 Finally, letting $p_j=\pi_n^{-1}(\zeta_j),j=0,\ldots,k+1,$ and using the estimate above, we obtain
\begin{align*}
d_{\S}(p,q)
\leq \sum_{j=0}^k d_{\S}(p_j,p_{j+1})
\leq \sum_{j=0}^k 3d_{\S_n}(\zeta_j,\zeta_{j+1})
= 3 d_{\S_n}(\pi_n(p),\pi_n(q)),
\end{align*}
which completes the proof.
\end{proof}

	%



\begin{lemma}\label{lemma:qc-extension}
Suppose there is an $\eta$-quasisymmetric mapping $f:\mathscr{S} \hookrightarrow \mathbb{R}^2$. Then there  are embeddings  $f_n:\mathscr{S}_n \hookrightarrow \mathbb{R}^2$ such that the following conditions hold:
\begin{itemize}
\item[(a)] $f_n|_{\Pi_{n}}$ is an $\eta_1$-quasisymmetric mapping for every $n$, where $\eta_1(t)=\eta(3t)$.  
\item[(b)] For all $n\geq 1$, $f_n$ is $H$-quasiconformal, where $H$ depends only on $\eta$.
\end{itemize}
\end{lemma}

\begin{remark}
It is possible to show that the mappings $f_n$ constructed below are in fact uniformly quasisymmetric on $\S_n$, however the details are not illuminating and we do not use this fact in the proof of Lemma \ref{lemma:necessary-condition-for-QS-embedding}.
\end{remark}

To prove Lemma \ref{lemma:qc-extension} we will need an extension result of Bonk, cf. Proposition 5.3 in \cite{Bonk}, which is a generalization of the classical Beurling-Ahlfors extension \cite{Ahlfors-Beurling}. 

\begin{theorem}\label{quasiextend}
	Let $D,D'\subset\mathbb{C}$ be Jordan domains and $f: \partial D \to \partial D'$ be an $\eta$-quasisymmetric mapping. Suppose that $\partial D$ is a $k$-quasicircle. If
	\[
	\min\{\mathrm{diam}(D), \mathrm{diam}(D')\} \leq \delta
	\]
	for some $\delta > 0$, then $f$ can be extended to an $\eta'$-quasisymmetric mapping $F: D \to D'$ where $\eta'$ only depends on $\delta, k$ and $\eta$.
\end{theorem}

The original theorem in \cite{Bonk} deals with Jordan regions in $\widehat{\mathbb{C}}$, however Theorem \ref{quasiextend} is easily obtained from Bonk's result.


\begin{proof}[Proof of Lemma \ref{lemma:qc-extension}]
 To define the embeddings $f_n:\S_n\hookrightarrow\mathbb{R}^2$ we will first define them locally on the lifts of (closed) dyadic squares $\D \subset [0,1]^2$ using Bonk's extension result above. The definition will be such that it will be consistent along the common parts of boundaries of such lifts in $\S_n$. 
	
For $n\geq 0$ and a dyadic square $ \Delta\in\mathcal{D}_{n}$ let $T=T_\Delta$ be the ``dyadic square" in $\S$ as before.
		
%
Observe that if $\Delta\in\mathcal{D}_{n}$ then $\Delta$ does not contain a slit of $S_n$ in its interior and hence the path metric on $\S_n$ restricted to $\pi_{n}(T)\subset\S_n$ coincides with the Euclidean metric on $\Delta=\pi_{0,n}(\pi_{n}(T))$.  Therefore $\pi_{n}(T)$ is isometric to a closed Jordan domain in $\mathbb{C}$ with the boundary which is a $\sqrt{2}$-quasicircle (since it is a square). On this boundary curve we define the following mapping 
\[
	f_n^{\partial T}:= f|_{\partial T} \circ (\pi_n|_{\partial T})^{-1} : \partial \pi_n(T) \to \mathbb{R}^2.
\]
Since $f$ is $\eta$-quasisymmetric and by Lemma \ref{lemma:-bilip-squares} $\pi_n|_{\partial T}^{-1}$ is $3$-bi-Lipschitz, it follows from  Lemma \ref{lemma:QSproperties} that $f_n^{\partial T}$ is an $\eta_1$-quasisymmetric map, where $\eta_1(t)=\eta(3t)$. Hence all the conditions of Theorem \ref{quasiextend} are satisfied and applying it to $f_n^{\partial T}$ and $\pi_n(T)$ we obtain that for every $\Delta\in\mathcal{D}_{n}$ there is a quasisymmetric map $f_n^{T} =f_n^{T_\Delta}  : \pi_n(T)  \to \mathbb{R}^2$
	which extents $f_n^{\partial T}$. Moreover, $f_n^{T}$ is $\eta_2$-quasisymmetric, where $\eta_2$ depends only on $\eta_1$, the quasiconformal constant of the boundary curves (i.e., $\sqrt{2}$ in this case), and diameters of these circles, which are bounded by $\textrm{diam}(\mathscr{S})$. Thus, $\eta_2$ is independent of  $n$ as well as of the particular dyadic square $\Delta\subset\mathcal{D}_{n}$ (or $T=T_\Delta$).	
	
	Combining the functions $f_n^{T}$ produces a homeomorphism $f_n: \mathscr{S}_n \to \mathbb{R}^2$. More precisely, if $\xi\in\S_{n}$ is such that $\xi \in \pi_n(T_\Delta)$ for some $\D\in\calD_{n}$ we let
	\begin{align}\label{extension}
	f_n(\xi)=  f_n^{T_\Delta}(\xi).
	\end{align}
 Note that $f_n$ is well defined since the squares $\{\pi_n(T_{\Delta})\}_{\Delta\in\mathcal{D}_{n}}$ cover $\mathscr{S}_n$ and the maps $f_n^{T_\Delta}$ coincide at points which are common to different dyadic squares of generation $n$ in $\S_{n}$.
 
	
		
To prove (a) note that $f_n |_{\Pi_{n}} = f \circ (\pi_n^{-1} |_{\Pi_{n}}).$ Since $f$ is $\eta$-quasisymmetric and $\pi_n^{-1}$ is $3$-bi-Lipschitz by Lemma \ref{lemma:-bilip-squares},  it follows that $f_n |_{\Pi_{n}}$ is $\eta_1$-quasisymmetric, where $\eta_1(t)=\eta(3t)$.		
		
For part $(b)$ note that $f_n$ is a homeomorphism, which is $\eta_2(1)$ - quasiconformal at every point $\xi\in\S_n$ such that $\pi_{0,n}(\xi)\in\mathrm{int}(\D)$ for some $\D\in\calD_{n}$. 

Next, suppose $\xi\in \Pi_{n}$. If $\xi$ does not belong to a slit then for $r$ sufficiently small  the ball $B(\xi,r)$ equipped with the metric $d_{\mathscr{S}_n}$ is isometric to the the same ball equipped with the Euclidean metric. Pick such an $r>0$ and denote by $\zeta_M$ and $\zeta_m$ the points at which the quantity $|f_n(\zeta)-f_n(\xi)|$ on the circle $\partial B(\xi,r) \subset\S_n$ is maximized and minimized, respectively.  Since  $f_n$ is a homeomorphism, we have
\begin{align*}
\frac{L_{f_n}(\xi,r)}{l_{f_n}(\xi,r)} 
&= \frac{\max\{|f_n(\zeta)-f_n(\xi)|: {d_{\S_n}(\zeta,\xi)\leq r\}}}{\min\{|f_n(\zeta)-f_n(\xi)|: {d_{\S_n}(\zeta,\xi)\geq r\}}}{=}\frac{|f_n(\zeta_M)-f_n(\xi)|}{|f_n(\zeta_m)-f_n(\xi)|} \\
&=\frac{|f_n(\zeta_M)-f_n(\xi)|}{|f_n(\zeta_M')-f_n(\xi)|} \cdot \frac{|f_n(\zeta'_m)-f_n(\xi)|} {|f_n(\zeta_m)-f_n(\xi)|}\cdot \frac{|f_n(\zeta'_M)-f_n(\xi)|}{|f_n(\zeta'_m)-f_n(\xi)|},
\end{align*}
where $\zeta'_M,\zeta'_m\in\Pi_{n}\cap \partial B(\xi,r)$ belong to the boundaries of the same $n$-th generation dyadic squares in $\S_{n}$ as $\zeta_M$ and $\zeta_m$. Therefore we have 
$H_{f_n}(\xi) \leq \eta_2(1)^2 \eta_1(1).$

\begin{figure}[tp]
	\centering
	\includegraphics[height=2in]{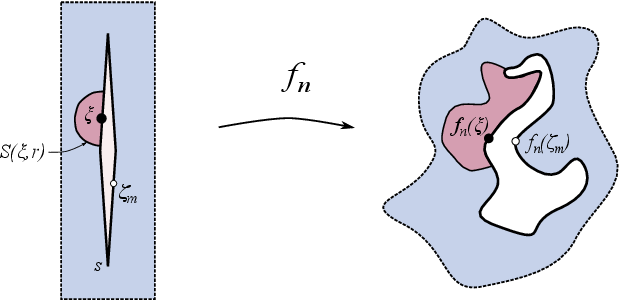}
	\caption{Estimating $l_{f_n}(\xi)=|f_n(\xi)-f_n(\zeta_m)|$ from below. A small neighborhood $A$ of the slit $s$ (drawn as a wedge) which is a topological annulus. Similarly, $A\setminus B(\xi,r)$ is also an annulus if $r>0$ is small enough. Therefore $\zeta_m$ belongs to the inner boundary component of the grey annulus on the left, i.e., $\zeta_m\in S(\xi,r)\cup (s\setminus B(\xi,r))$.}\label{fig:qc-on-slit}
\end{figure}	

If $\xi\in\Pi_n$ then we may take $r>0$ small enough so that $B(\xi,r)$ does not intersect any other slits. We also denote $S(\xi,r)=\{\zeta\in \S_n : d_{\S_n}(\zeta,\xi)=r\}$.

Next, just as above let $\zeta_M$ and $\zeta_m$ be such that 
\begin{align*}
L_{f_n}(\xi,r)=|f_n(\zeta_M)-f_n(\xi)|
\quad \mbox{and} \quad l_{f_n}(\xi,r)&=|f_n(\zeta_m)-f_n(\xi)|.
\end{align*}

Since $f_n$ is a homeomorphism, then either $|\zeta_M-\xi|=r$, or $|\zeta_M-\xi|\leq r$ and $\zeta_M\in \Pi_n$. In either case, there is a point $\zeta_M'\in \Pi_n$ on the boundary of the dyadic square containing $\zeta_M$, such that $|\zeta_M'-\xi|=r\geq |\zeta_M-\xi|$.
Therefore,
\begin{align}\label{est:distortion1}
L_{f_n}(\xi,r)=|f_n(\zeta_M)-f_n(\xi)|\leq \eta_2(1) |f_n(\zeta_M')-f_n(\xi)|.
\end{align}
where $\zeta_M'\in\Pi_n\cap S(\xi,r)$.

To estimate $l_{f_n}(\xi,r)$, note that  $\zeta_m\in S(\xi,r)\cup (s\setminus B(\xi,r))$, since $f_n$ is a homeomorphism. Indeed, if $r$ is small enough then there is  a neighborhood $A$ of the slit $s$ which is a topological annulus such that $B(\xi,r)\subset A$, see Figure  \ref{fig:qc-on-slit}. Therefore, $\zeta_m$ belongs to the inner boundary of the annulus $A\setminus B(\xi,r)$, which may be written as $S(\xi,r)\cup (s\setminus B(\xi,r))$. If $\zeta_m\in S(\xi,r)$ then there is a point $\zeta_m'\in S(\xi,r)\cap\Pi_n\cap T$, where $T$ is a dyadic square containing $\zeta_m$. Therefore
\begin{align}\label{est:distorion2}
|f_n(\zeta_m)-f_n(\xi)| \geq \frac{1}{\eta_2(1)} |f_n(\zeta_m')-f_n(\xi)|.
\end{align}
On the other hand, if $|\zeta_m-\xi|>r$ then $\zeta_m \in s$, and choosing any point $\zeta_m'\in s \cap S(\xi,r)$ we obtain
 \begin{align}\label{est:distortion3}
|f_n(\zeta_m)-f_n(\xi)| \geq \frac{1}{\eta_1(1)} |f_n(\zeta_m')-f_n(\xi)|.
\end{align}

Combining  (\ref{est:distortion1}),(\ref{est:distorion2}), and (\ref{est:distortion3}) we obtain
\begin{align*}
\frac{L_{f_n}(\xi,r)}{l_{f_n}(\xi,r)} 
&\leq \max(\eta_2^2(1),\eta_1(1)\eta_2(1))\frac{|f_n(\zeta_M')-f_n(\xi)|}{|f_n(\zeta_m')-f_n(\xi)|}\\
&\leq \max(\eta_1(1)\eta_2^2(1),\eta_1^2(1)\eta_2(1)),
\end{align*} 
where the last inequality holds because $\zeta_M',\zeta_m'\in\Pi_n\cap S(\xi,r)$ and $f_n$ is $\eta_1$-quasisymmetric on $\Pi_n$ by part $(a)$.
\end{proof}

\begin{proof}[Proof of Lemma \ref{lemma:necessary-condition-for-QS-embedding}]
Assume there exists a quasisymmetric embedding $f:\S\to\mathbb{R}^2.$ 
%
%
%
%
By Lemma \ref{lemma:qc-extension} there exists an $H$-quasiconformal map $f_n:\S_n\to\mathbb{R}^2$ such that 
\begin{align*}
f_n(\S_n)=\overline{\Omega_n' \backslash \left( \bigcup_{j=0}^{m_n-1} K_{n,j}'\right) }\subset\mathbb{R}^2
\end{align*}
 where $m_n=1+4+\ldots+4^{n-1}$, and $\{\mathrm{int}(K_{n,j}')\}$ are pairwise disjoint Jordan domains compactly contained in the Jordan domain $\Omega_n'$. Moreover, since $f_n|_{\Pi_n}$ is $\eta'$-quasisymmetric, with $\eta'$ depending on $\eta$, we have that $\Omega_n'$ and $K_{n,j}'$'s are all $k$-quasidisks, with $k$ depending only on $\eta$.
We denote $\mathcal{K}_n'=\{{K_{n,0}'},\ldots, {K_{n,m_n-1}'}\}$ and $K_n'=\cup_j K_{n,j}'$.

Observe that we may assume that the image of the ``outer square" of $\S_n$ under $f_n$ is the ``outermost" boundary component $\g_n'$ of $\Omega_n'$, i.e., $\g_n'= f_n(\pi_{0,n}^{-1}(\partial{\mathbb{U}}))$ is the boundary of the unbounded component of $\mathbb{R}^2\setminus \Omega_n'$. This can be achieved by post-composing $f_n$ with an appropriate M\"obius transformation of the plane and denoting the resulting mapping by $f_n$ again. Indeed, we may first post-compose $f_n$ with a scaling so that $\diam(f_n(\pi_{0,n}^{-1}(\partial{\mathbb{U}})))=1$. Then, one can compose the result with a reflection in the boundary of a largest disk, of radius say $\alpha_n$, inscribed in the domain bounded by $f_n(\pi_{0,n}^{-1}(\partial{\mathbb{U}}))$. Since $f_n(\pi_{0,n}^{-1}(\partial{\mathbb{U}}))$ is a $k'$-quasicircle, and therefore is quasi-round, the radius $\alpha_n$ of the disk is bounded from below by a constant depending only on $k'$.   Therefore, the resulting mapping will be uniformly quasisymmetric on $f_n(\S_n)$, since by (\ref{diam}) we have
 $$
 {\diam \left(f_n(\S_n)\right)} \leq \eta'(2)\cdot{\diam(f_n(\pi_{0,n}^{-1}(\partial{\mathbb{U}})))}=\eta'(2).$$

Let $L_n=\pi_{0,n}^{-1}(L)$ and $R_n=\pi_{0,n}^{-1}(R)$ be the ``vertical sides"  of $\mathscr{S}_n$. Denote
\begin{align*}
\G_n:=(f_n(L_n),f_n(R_n);\Omega_n').
\end{align*} 
Thus, $\G_n$ is the family of curves in $\Omega_n'$ connecting the images of lifts of the vertical sides of the unit square $\mathbb{U}$ under $f_n$. Next we observe that for all $n\geq 1$  we have 
\begin{align}\label{trmod-estimate1}
\Mod_{\mathbb{U},\calK_n}(\G) \geq H^{-1} \Mod_{\Omega_n',\calK_n'}(\G_n).
\end{align}
Indeed,  the identity map $id_n$ from $S_n=\mathbb{U}\setminus K_n$ equipped with the Euclidean metric to $S_n$ with the inner metric $d_{S_n}$ is a local isometry and therefore is $1$-quasiconformal. Hence, by letting $\phi_n:=f_n \circ id_n$, we have that
$\phi_n : {\mathbb{U}}\setminus K_n \to \Omega_n'\setminus K_n'$
is an $H$-quasiconformal map between domains in $\mathbb{R}^2$. Moreover, the mapping $\phi_n$  descends to a homeomorphism between the quotient spaces
$\tilde{\phi}_n:(\mathbb{U})_{\calK_n} \to (\Omega_n')_{\mathcal{K}_n'},$
and if $\tilde{\G}_{}$ and $\tilde{\G}_n$ are the images of $\G_{}$ and $\G_n$ under the quotient maps, then
$\tilde{\phi}_n(\tilde{\G}_{})=\tilde{\G}_n.$
Therefore, by (\ref{trmod-quotients}) and Lemma \ref{lemma:mod-comparison} we have that
\begin{align*}
\Mod_{\mathbb{U},\calK_n}(\G) = \Mod_{\mathbb{U},\calK_n}(\tilde{\G}) \geq H^{-1} \Mod_{\Omega_n',\calK_n'}\left(\tilde{\phi}_n(\tilde{\G})\right)=H^{-1} \Mod_{\Omega_n',\calK_n'} (\G_n).
\end{align*}

Since $K_{n,j}'$'s are $k$-quasidisks, by Lemma \ref{lemma:mod-comparison} we have that 
\begin{align}\label{trmod-estimate2}
\mathrm{Mod}_{\Omega_n',\mathcal{K}_n'} (\G_n) \geq \min \{ c_1, c_2  \mod(\G_n)\}.
\end{align}
Moreover, since $\Omega_n'$ is a $k$-quasidisk it is then Loewner (see, e.g., \cite[Proposition 7.3]{Bonk}) and therefore 
\begin{align}\label{trmod-estimate3}
\mod (\G_n) \geq \psi (\Delta(f_n(L_n),f_n(R_n))),
\end{align}
where $\psi$ depends only on $k$ (and therefore on $\eta$). However, if $x_n\in L_n$ and $y_n\in R_n$ are such that $\dist(f_n(x_n),f_n(y_n)) = \dist(f_n(L_n),f_n(R_n))$ then 
\begin{align}\label{reldist-estimate}
\Delta(f_n(L_n),f_n(R_n)) \geq  \frac{\dist(f_n(x_n),f_n(y_n))}{\diam f_n(L_n\cup R_n) } \geq \frac{1}{2\eta \left(\frac{\diam (L_n\cup R_n)}{\dist(x_n,y_n)}\right)}
\geq \frac{1}{2\eta(2)},
\end{align}
since $\dist(x_n,y_n)\geq 1$ and $\diam (L_n\cup R_n) \leq \diam \mathscr{S}_n \leq 2$.

From (\ref{trmod-estimate1}),(\ref{trmod-estimate2}), (\ref{trmod-estimate3}), and (\ref{reldist-estimate}) it follows that for all $n\geq 1$ we have 
\begin{align*}
\Mod_{\mathbb{U},\mathcal{K}_n} (\G) \geq H^{-1} \min\left\{ c_1,c_2\psi \left( \frac{1}{2\eta(2)} \right) \right\},
\end{align*}
where $H,c_1,c_2$ and $\psi$ depend only on $\eta$.
\end{proof}

%
%
%
%


\section{Upper bounds}\label{Section:non-embedding}

In this section we prove the ``only if" direction of Theorem \ref{thm:main}. For this we estimate the transboundary modulus of curve families connecting the vertical sides of the unit square in dyadic slit domains. In particular we show that if the sequence of relative sizes $r_i$ of slits is not square summable then the transboundary modulus approaches $0$. Combining with the results of Section \ref{Section:necessary-condition} we show that if $\mathbf{r}\notin\ell^2$ then there is no quasisymmetric embedding of $\S_{\mathbf{r}}$ into the plane, cf. Theorem \ref{thm:necessity}.

\subsection{Estimates for Transboundary Modulus in slit domains }\label{Section:trmod-in-slit-domains} The following lemma is the main result of this section. Below we use the same notation as in Section \ref{Section:necessary-condition}. 

\begin{lemma}\label{lemma:transmod=0}
Let $\G$ be the collection of all the curves in the unit square $[0,1]^2$ connecting the vertical edges of the square.  Suppose $\mathbf{r}=\{r_i\}_{i=0}^{\infty}$ is a sequence of numbers in $(0,1)$ such that $\mathbf{r}\notin\ell^2$. Then for every $0<\epsilon<1$ we have
\begin{equation}\label{trmod}
	\Mod_{\mathbb{U},\mathcal{K}_n}(\Gamma) \leq \prod_{i=0}^{n}(1-\frac{1}{8}\epsilon r_i^2)+ 3\epsilon,
	\end{equation}
for every $n\geq 0$.
In particular, if $\{r_i\}_{i=0}^{\infty}\notin\ell^2$ then
\begin{align}\label{limit:trmod0}
  \lim_{n\to\infty}\Mod_{\mathbb{U},\calK_n} (\G)=0.
\end{align}
\end{lemma}

Before proceeding to the proof we observe that by combining Lemma \ref{lemma:transmod=0} with Lemma \ref{lemma:necessary-condition-for-QS-embedding} we obtain the necessity in Theorem \ref{thm:main}.

\begin{corollary}\label{thm:necessity}
If $\mathbf{r}\notin\ell^2$ then there is no quasisymmetric embedding of $\S=\S_{\mathbf{r}}$ into the plane $\mathbb{R}^2$. 
\end{corollary}

\begin{proof}[Proof of Lemma \ref{lemma:transmod=0}]
The proof below is similar to proofs in \cite{Hakobyan}, where  estimates for the classical modulus in slit domains were obtained. However, transboundary modulus in general can be larger than the classical modulus and therefore the results in this section do not follow directly from \cite{Hakobyan}. 

\subsubsection{Constructing mass distribution $\varrho_n^{\epsilon}$} We will first prove the estimate (\ref{trmod}) assuming that the sequence $r_i$ is such that for every $i\geq 0$ we have $r_i=2^{-j_i}$ for some $j_i\geq 1$, and $\epsilon=2^{-m}$ for some $m\geq 1$. The estimate is obtained by defining a particular mass distribution for the pair $(\mathbb{U}, \mathcal{K}_n)$. In order to do that, new notations are introduced below.

Let  $s=s(\D)=\{x\}\times [a,b] \subset \mathbb{U}$ be a slit of length $l(s)=b-a$. For $\epsilon\in(0,1)$ the \emph{$\epsilon$-collar} of $s$ is the rectangle $s^\epsilon =  (x,x+\epsilon l(s)) \times s$. Equivalently,
\[
s^\epsilon=s+(0,\epsilon l(s))=\{t+x : t \in s, x\in (0,\epsilon l(s))\}.
\]

Let ${t}(s^\epsilon), b(s^\epsilon), \ell(s^\epsilon), r(s^\epsilon)$ be the top, bottom, left, and right sides of $s^\epsilon$, respectively. Note that $\ell(s^\epsilon)=s$. 

\begin{lemma}\label{lemma:disjoint}
Assume that $r_n=1/2^{j_n}, n\geq 1,$ and $\epsilon=2^{-m}$  for some natural numbers $\{j_n\}_{n=1}^{\infty}$ and $m\geq1$. 
Then the $\epsilon$-collars of any two slits $s$ and $s'$ are either disjoint, or one is completely contained in the other.
\end{lemma}

\begin{proof}
If $s=s(\D)=\{x\}\times[a,b]$ with $\D\in\calD_n$, $r_n=2^{-j_n}$ and $\epsilon=2^{-m}$ then $\overline{s^{\epsilon}}$ is a rectangle that can be written as a union of $\epsilon^{-1}=2^m$ dyadic squares of generation $N=n+j_n+m$. Therefore, if $\D'$ is a dyadic subsquare of $\D$ of generation $k\geq N$ then it is either disjoint from $\overline{s^{\epsilon}}$ or is completely contained in it and the same is true for $s'=s(\D')$. On the other hand, if $\D'$ is a dyadic square of generation $k\leq N-1$ in $\D$, and $s'=s(\D')=\{x'\}\times[a',b']$, then  the distance between $x$ and $x'$ is at least a half of the sidelength of $\D'$ and therefore
\begin{align*}
|x-x'|\geq \frac{1}{2} 2^{-k} \geq 2^{-1-(N-1)} =2^{-N}.
\end{align*}
Since the width of $s$ is exactly $2^{-N}$ and $(s')^{\epsilon}$ is located to the right of the slit $s'$, it follows that the $\epsilon$-collars of $s$ and $s'$ are disjoint if $x'>x$. In the case $x'<x$ there is nothing to prove since any dyadic square $\D'$ contained in the left half of $\D$ does not intersect $s^{\epsilon}$.
\end{proof}

%
%
%
%
%

\begin{figure}[tp]
	\centering
	\includegraphics[height=1.5in]{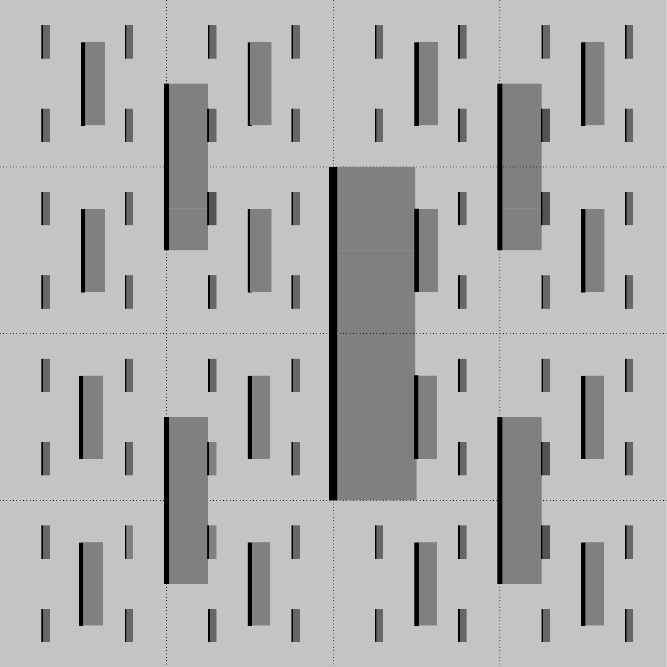}
	\caption{Choosing a subset $\mathcal{K}_{\epsilon}$ from the family of slits $\mathcal{K}$ so that the $\epsilon$-collars $v_j^{\epsilon}$ (darker grey rectangles) are disjoint. Note that the slits within the collars are discarded (not included in $\mathcal{K}_{\epsilon}$). In this example $\epsilon=1/4$ and $r_i=1/2$ for $i=0,1,2,3$.}
\end{figure}

From the above it follows that it is possible to select an infinite subsequence $\calK_{\epsilon}=\{s_{i_n}\}$ in $\mathcal{K}$ for which the $\epsilon$-collars are disjoint (i.e., the ``smaller" collars which are contained in ``larger" ones are not enumerated). Indeed, we may first enumerate $\mathcal{K} = \{s_i\}_{i=0}^\infty$ so that the lengths of the slits are non-increasing, i.e., $l(s_i) \geq l(s_{i+1})$ for every $i \geq 0$. Then, we choose the sequence $v_n:=s_{i_n}$ by induction. Let $v_0 =s_0$. Suppose for $n \geq 1$ the sequence  $v_0, \ldots, v_{n-1}$ has been defined, and 
%
%
let  $v_n = s_{i_n}$, where
\begin{align*}
i_n = \min \left\{ j :  s_j^\epsilon \cap \left( \bigcup_{i}^{n-1}v_i^\epsilon \right) = \emptyset \right\}.
\end{align*}
Since the set $[0,1]^2\setminus\left(\bigcup_i^{n-1} v_i^{\epsilon}\right)$ always contains a dyadic square (it has a nonempty interior), the process never ends and  the collars $\{v_i^\epsilon\}_{i=0}^\infty$ are disjoint by  construction. Let 
\begin{align*}
\mathcal{K}_\epsilon = \{v_i\}_{i=0}^\infty
\end{align*}
denote this subsequence. Moreover, for $n\geq 0$ let
\begin{align*}
\mathcal{K}_{\epsilon,n} &= \mathcal{K}_{\epsilon} \cap \mathcal{K}_n= \{v_i \}_{i=0}^{N_\epsilon},
\end{align*} 
where
%
%
%
$N_\epsilon = |\mathcal{K}_{\epsilon} \cap \mathcal{K}_n|$ is the cardinality of $\calK_{\epsilon,n}$. 


\begin{figure}[tp]
	\centering
	\includegraphics[height=1.5in]{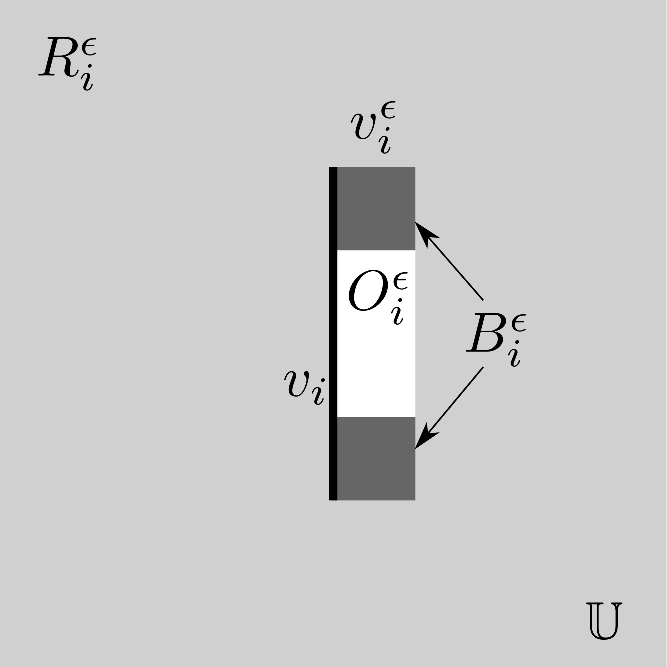}
	\quad
	\includegraphics[height=1.5in]{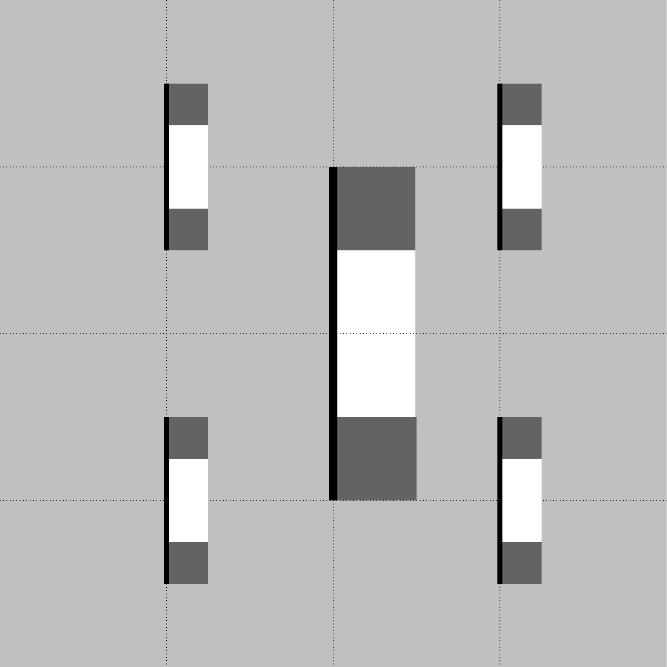} \quad
	\includegraphics[height=1.5in]{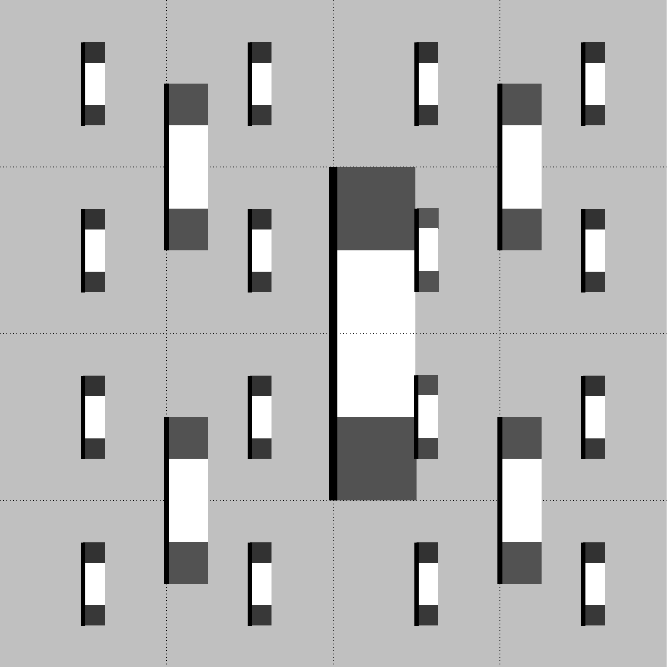}
	\caption{The white, dark grey and light grey regions in the picture on the left are the $\epsilon$-omitted, buffer, and residual subsets corresponding to a slit $v_i$ in the unit square $\mathbb{U}$. The second and third pictures show the sets $\mathcal{O}_n^{\epsilon}, \mathcal{B}_n^{\epsilon}$ and $\mathcal{R}_n^{\epsilon}$ (in white, dark grey and light grey, respectively) for the "standard" collection of slits corresponding to the sequence $r_i=1/2, i\geq0$. Here $\epsilon=1/4$, and $n=1,2$.  }
\end{figure}	

For $\epsilon$ as above, we denote by $B_i^\epsilon$,  the \emph{$\epsilon$-buffer} of the slit $v_i$, the union of the top and bottom squares in $v_i^\epsilon$. More precisely,
\[
B_i^\epsilon=\{x \in v_i^\epsilon : \mathrm{dist}(x,t(v_i^\epsilon)) \leq \epsilon l(v_i) \ \mathrm{or} \ \mathrm{dist}(x,b(v_i^\epsilon)) \leq \epsilon l(v_i)\}.
\]
The sets $O_i^\epsilon= v_i^\epsilon \backslash B_i^\epsilon$ and $R_i^\epsilon= \mathbb{U} \backslash v_i^\epsilon = \mathbb{U} \backslash (B_i^\epsilon \cup O_i^\epsilon)$ will be called the \emph{$\epsilon$-omitted and residual regions of $v_i$}, respectively. 
%

We also define the \emph{$\epsilon$-buffer, omitted and residual sets in $\mathbb{U}$}, denoting them by $\mathcal{B}_n^{\eps}, \mathcal{O}_n^\epsilon,\mathcal{R}_n^\epsilon$, respectively, as follows:
\begin{align}\label{notation:buffers}
\mathcal{B}_n^\epsilon &= \bigcup_{v_j \in \mathcal{K}_{\epsilon,n}} B^\epsilon_j,
\qquad 
\mathcal{O}_n^\epsilon = \bigcup_{v_j \in \mathcal{K}_{\epsilon,n}} O^\epsilon_j,
\qquad
\mathcal{R}_n^\epsilon = \mathbb{U} \backslash (\mathcal{B}_n^\epsilon \cup \mathcal{O}_n^\epsilon).
\end{align}

Note that the $O_{i}^{\epsilon}$ and $\mathcal{O}_n^{\epsilon}$ are both open sets, while $\mathbb{U}\setminus\mathcal{O}_n^{\epsilon}$ is a compact set for every $n\geq 1$.

Finally, we define a Borel function $\rho_n^\epsilon: \mathbb{U} \backslash K_n \to [0,\infty]$ and  weights $\{\rho_{n,j}\}=\{\rho_{n}(s_j)\}$ on $\mathcal{K}_n$ as follows:
\begin{align}\label{mass-distribution}
\begin{split}
\rho_n^\epsilon &
:= \chi_{\mathcal{B}_n^\epsilon \cup \mathcal{R}_n^\epsilon}= 
\chi_{\mathbb{U} \backslash \mathcal{O}_n^\epsilon} = 
\begin{cases}
0, &   \mbox{ on }  \mathcal{O}_n^{\epsilon},\\
1, &   \mbox{ on }  \mathcal{B}_n^{\epsilon} \cup \mathcal{R}_n^{\epsilon}.
\end{cases}
\\
\rho_{n,j} &:= \rho_n(s_j) =
\begin{cases}
\epsilon l(s_j), &   s_j\in \mathcal{K}_{\epsilon,n},\\
0, &   s_j\in \calK_n\setminus\mathcal{K}_{\epsilon,n}.
\end{cases}
\end{split}
\end{align}
%
where $\chi_{E}$ denotes the characteristic function of the set $E$, and let
\begin{align*}
\varrho^{\epsilon}_n = (\rho_n^\epsilon; \rho_{n,1}, \ldots, \rho_{n,N}),
\end{align*}
where $N=1+\ldots+4^n$ is the number of slits of generation at most $n$. In other words, $\rho_n^{\epsilon}$ vanishes on the omitted set and is equal to $1$ otherwise, while $\rho_{n,j}$ is equal to the width of the $\epsilon$-collar for each slit $s_j\in \mathcal{K}_{\epsilon,n}$ and is $0$ otherwise.

\subsubsection{Admissibility of $\varrho_{n}^{\epsilon}$ relative $\calK_n$.}

Next, we show that $\varrho^{\epsilon}_n$ is admissible for $\G$ relative $\calK_n$, i.e., the estimate
\begin{align}\label{trmod-admissible}
l_{\varrho_n^{\epsilon}}(\g) = \int_{\g} \rho_n^{\epsilon} ds + \sum_{\gamma \cap s_i \neq \emptyset} \rho_{n,i} \geq 1,
\end{align}
holds for every $\g\in\G$.

In \cite{Hakobyan} it was shown that if $\g\in\G$ does not intersect any of the slits of $\calK_n$ then $\rho_{n}^{\epsilon}$-length of $\g$ (i.e., $\int_{\g}\rho_n^{\epsilon}$) is at least $1$. The idea and the reason for defining the discrete weights $\rho_{n,j}$ as in (\ref{mass-distribution}), is to ensure that when a curve $\g\in\G$ intersects a slit $s_j\in\calK_n$, its ``horizontal-length" does not decrease too much. Indeed, if $\g$ intersects a slit $s_i$ the integral $\int_{\g}\rho_n^{\epsilon} ds$ may decrease by the amount equal to the width of the corresponding collar (or more), but the second term in $l_{\varrho_n^{\epsilon}}(\g)$ would increase by $\rho_{n,j}=\epsilon l(s_j)$, which is the ``width" of the collar of $s_{i}^{\eps}$. This balance implies that the $\varrho_n^{\epsilon}$-lengths of the curves stays bounded below by $1$. Next we provide the details of this argument.

To prove (\ref{trmod-admissible}) we will show that for every $\g\in\G$ there is a subset $\g'\subset\mathbb{U}$, which is not necessarily a curve, such that
\begin{align*}
l_{\varrho_n^{\epsilon}}(\g)\geq l_{\varrho_n^{\epsilon}}(\g') \mbox{ and }  \l_{\varrho_n^{\epsilon}}(\g')\geq 1.
\end{align*}


Pick a curve $\g\in\G$. Without loss of generality, we may assume that  $\gamma$ is oriented so that it starts at the left and ends at the right vertical edge of the unit square $\mathbb{U}$. Given two disjoint subsets $E$ and $F$ in $\mathbb{U}$, we say that \emph{$\gamma$ meets $E$ before $F$} if there exists $t\in(0,1)$ so that $\gamma(t) \in E$ and $\gamma(s) \notin F$ for any $s < t$ and \emph{$\gamma$ meets $E$ after $F$} if $\gamma$ meet $F$ before $E$. Before constructing $\g'$, we modify $\g$ inductively around every slit $v_i\in\calK_{\epsilon,n}$ as described next.
	
	
Denote $\g_{-1}:=\g$. For $0\leq i\leq N_{\epsilon}$, suppose the subsets $\g_0,\ldots,\g_{i-1}$ of $\mathbb{U}$ have been defined and define $\g_i$ as follows:
	\begin{itemize}
	\item[(a)] If $\gamma\cap v_i = \emptyset$, then
	\begin{align*}
	\gamma_i=\left\{
	\begin{array}{ll}
	\gamma_{i-1} & \mathrm{if} \ \gamma \cap O_i^\epsilon = \emptyset,
	\\
	(\gamma_{i-1} \backslash v_i^\epsilon) \cup (t(v_i^\epsilon)\setminus v_i) & \mathrm{if} \ \gamma \ \mathrm{meets} \ O_i^\epsilon \ \mathrm{before} \ r(v_i^\epsilon),
	\\
	\gamma_{i-1} \backslash O_i^\epsilon & \mathrm{if} \ \gamma \ \mathrm{meets} \ O_i^\epsilon \ \mathrm{after} \ r(v_i^\epsilon).
	\end{array}\right.
	\end{align*}
	\item[(b)] If  $\gamma \cap v_i \neq \emptyset$ then
		\begin{equation*}
	\gamma_i = (\gamma_{i-1} \backslash (v_i^\epsilon \cup v_i)) \cup (t(v_i^\epsilon)\setminus v_i),
	\end{equation*}
	\end{itemize}
where $t(v_i^{\epsilon})$ and $r(v_i^{\epsilon})$ as before denote the top and the right sides of the collar $v_i^{\epsilon}$, respectively. 
	
	This is a finite induction. Thus, we only construct $\gamma_i$ for $i= 0, \ldots, N_\epsilon$ and let $$\gamma' =\gamma_{N_\epsilon}.$$
 Note that $\gamma'\subset\mathcal{B}_n^\epsilon \cup \mathcal{R}_n^\epsilon$. Moreover, at every step of the construction above the curves are modified so that the projection of $\g_i$ to the $x$-axis is equal to the interval $[0,1]$ except for possibly finitely many dyadic points. Thus, we have  $\calH^1(\mathrm{proj}_x(\g_i))=1$, where $\mathrm{proj}_x$ denotes the projection onto the $x$-axis in the plane. By induction, we also have  $\calH^1(\mathrm{proj}_x(\gamma')) = 1$.  Therefore
	\begin{equation*}
	l_{\varrho_n^\epsilon}(\gamma') = \int_{\g'}{\rho_n^{\epsilon}}ds =\mathcal{H}^1(\gamma') \geq \mathcal{H}^1(\mathrm{proj}_x(\gamma')) = \mathcal{H}^1([0,1])=1,
	\end{equation*}
%
and it would be sufficient to prove that $l_{\varrho_n^\epsilon}(\gamma) \geq l_{\varrho_n^\epsilon}(\gamma')$. 	Since $\g=\g_{-1}$ and $\g'=\g_{N_{\epsilon}}$, it is enough to show that for every $0 \leq i \leq N_\epsilon$ we have
\begin{align}\label{lengths-decreasing}
l_{\varrho_n^\epsilon}(\gamma_{i-1}) \geq l_{\varrho_n^\epsilon}(\gamma_i).
\end{align}

\begin{figure}[tbp]
		\centering
		\includegraphics[height = 1.5in]{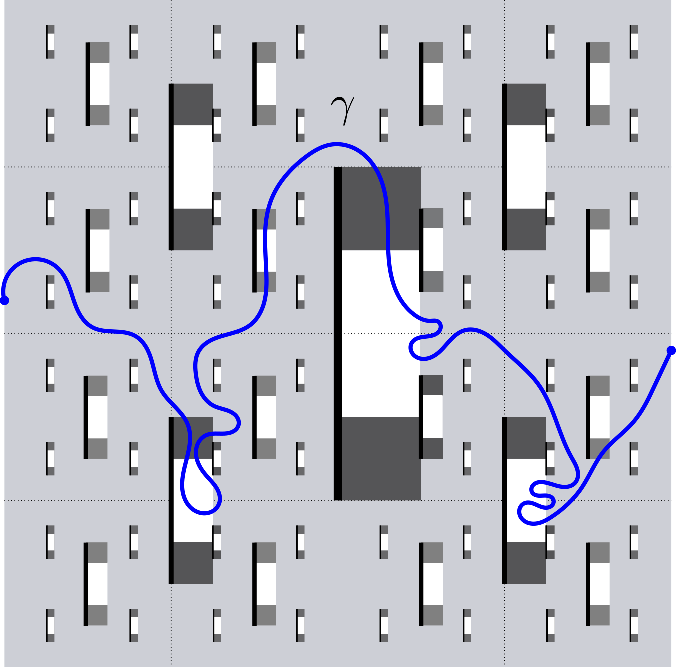}
		\includegraphics[height = 1.5in]{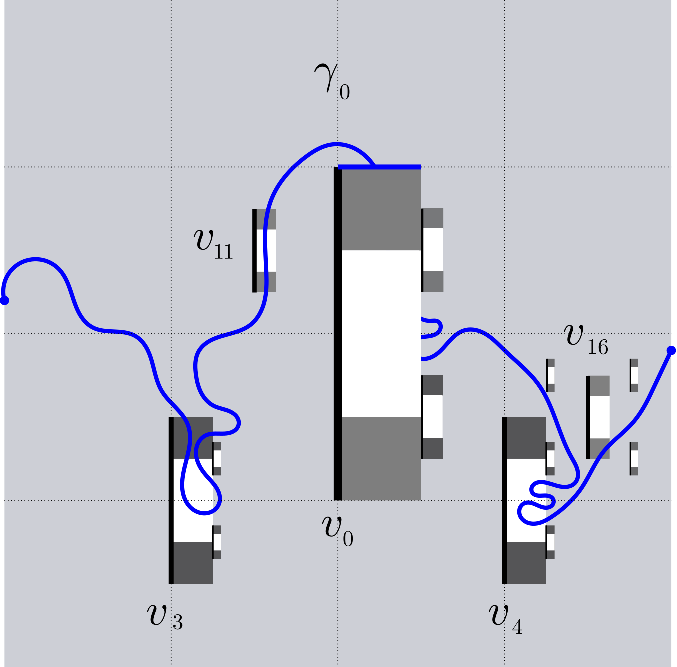}\label{newcurves}
			\includegraphics[height = 1.5in]{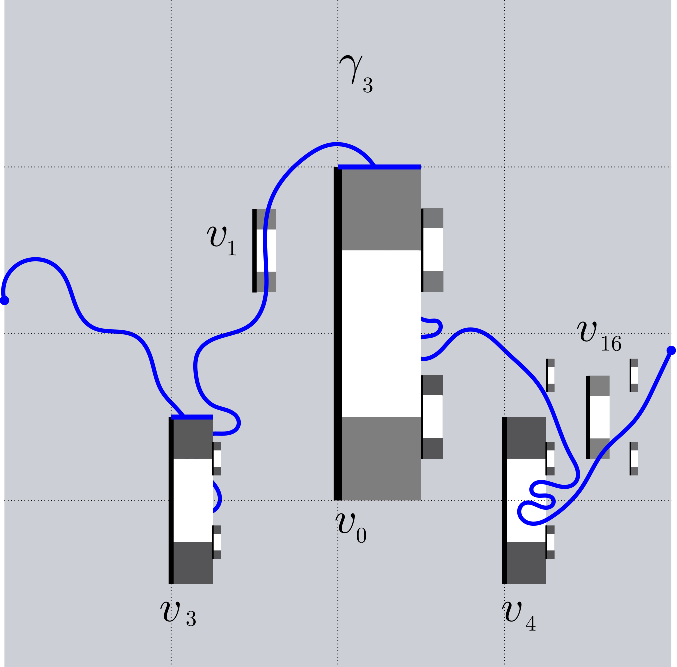}\quad
		\caption{A curve $\g=\g_{-1}$ and its modifications. Since $\g$ meets the slit $v_0$, we have $\g_{0}=\g\setminus v_0^{\epsilon} \cup (t(v_0^{\epsilon})\setminus v_0)$. Since $\g$ does not intersect the collars of $v_1$ and $v_2$, we have $\g_2=\g_1=\g_0$. Since $\g_2$ meets $O_3$ before $r_3$, we have  $\g_3=\g_2\setminus v_3^{\epsilon} \cup (t(v_3^{\epsilon})\setminus v_3)$.}\label{fig:modifications}
	\end{figure}

%
	
	By the definition of mass distribution $\varrho_n^{\epsilon}$ in (\ref{mass-distribution}), we have
	\begin{eqnarray}
		l_{\varrho_n^\epsilon}(\gamma_{i-1}) & = & \mathcal{H}^1(\gamma_{i-1} \cap \mathcal{R}_n^\epsilon) + \mathcal{H}^1(\gamma_{i-1} \cap \mathcal{B}_n^\epsilon) +\sum_{{\{j \,: \, \gamma_{i-1} \cap v_j \neq \emptyset\} } } \rho_{n,j} \nonumber
		\\
		& = &\mathcal{H}^1(\gamma_{i-1} \cap \mathcal{R}_n^\epsilon) + \sum_{j=0}^{N_\epsilon}\mathcal{H}^1(\gamma_{i-1} \cap B_j^\epsilon) +\sum_{\{j\, :\,  \gamma_{i-1} \cap v_j \neq \emptyset\} }  \rho_{n,j}\nonumber.
	\end{eqnarray}
Therefore, letting
\begin{align*}
\delta_{i,j}=
\begin{cases}
1, \mbox{ if } \g_i\cap v_j \neq \emptyset,\\
0, \mbox{ otherwise,}
\end{cases}
\end{align*}
we have
\begin{align}\label{length-equality}
l_{\varrho_n^\epsilon}(\gamma_{i-1})=\mathcal{H}^1(\gamma_{i-1} \cap \mathcal{R}_n^\epsilon) + \sum_{j=1}^{N_\epsilon}\left(\mathcal{H}^1(\gamma_{i-1} \cap B_j^\epsilon) +\delta_{i-1,j} \cdot \rho_{n,j}\right).
\end{align}
	
Since $\gamma_i$ is obtained by modifying $\g_{i-1}$ only within $\overline{(v_i^{\eps})}$, we have that the two curves coincide on the residual set $\mathcal{R}_n^{\eps}$ (note that $t(v_i^{\eps})$ is in the complement of  $\mathcal{R}_n^{\eps}$), and therefore 
\begin{align}\label{length-in-residual}
\calH^1(\gamma_{i-1} \cap \mathcal{R}_n^\epsilon) = \mathcal{H}^1(\gamma_i \cap \mathcal{R}_n^\epsilon),
\end{align}
and for every $j\in\{0,\ldots,N_{\epsilon}\}$ with $j\neq i$ we have
\begin{align}\label{length-buffers-equality}
\mathcal{H}^1(\gamma_{i-1} \cap B_j^\epsilon) +\delta_{i-1,j} \cdot \rho_{n,j} = \mathcal{H}^1(\gamma_{i} \cap B_j^\epsilon) +\delta_{i,j} \cdot \rho_{n,j}.
\end{align} 	
Therefore, by (\ref{length-equality}) and since $\delta_{i,i}=0$,  to prove (\ref{lengths-decreasing}) we only need to show the following estimate
\begin{align}\label{length-in-collar}
\mathcal{H}^1(\gamma_{i-1} \cap B_i^\epsilon) +\delta_{i-1,i} \cdot \rho_{n,i} 
\geq\mathcal{H}^1(\gamma_{i} \cap B_i^\epsilon).
\end{align} 	
Corresponding to the definition of $\varrho_n^{\epsilon}$ in (\ref{mass-distribution}), there are several cases to consider:
\begin{itemize}
\item[(a)] If $\gamma_{i-1}\cap v_i =\emptyset$, i.e., $\delta_{i-1,i}=0$, then three possibilities can occur:
\begin{itemize}
  \item[-] If $\g\cap O_i = \emptyset$ then $\g_{i-1}\cap B_i^{\epsilon} = \g_i\cap B_i^{\epsilon}$. In particular $\calH^1(\g_{i-1}\cap B_i^{\epsilon}) =
      \calH^1(\g_i\cap B_i^{\epsilon})$.
  \item[-] If $\g$ meets $O_i^{\epsilon}$ before $r(v_i^{\epsilon})$ then $\g_{i-1}$ connects the top and bottom of one component of an $\epsilon$-buffer and therefore $\calH^1(\g_{i-1}\cap B_i^{\epsilon}) \geq \epsilon l(v_i) = \calH^1(\g_i\cap B_i^{\epsilon}).$
  \item[-] If $\g$ meets $O_i^{\epsilon}$ after $r(s_i^{\epsilon})$ then $\calH^1(\g_{i-1}\cap B_i^{\epsilon}) = \calH^1(\g_i\cap B_i^{\epsilon})$.
\end{itemize}
\item[(b)] If $\gamma_{i-1}\cap v_i \neq \emptyset$ then
\begin{align*}
\mathcal{H}^1(\gamma_{i-1} \cap B_i^\epsilon) + \rho_{n,i}
\geq \rho_{n,i}=\epsilon l(s_i)= \calH^1(t(v_i^\epsilon)) = \mathcal{H}^1(\gamma_i \cap B_i^\epsilon).
\end{align*}
\end{itemize}
Thus (\ref{length-in-collar}) holds in all the cases. Combining (\ref{length-equality}),(\ref{length-in-residual}),(\ref{length-buffers-equality}) and (\ref{length-in-collar}) we obtain (\ref{lengths-decreasing}).
Therefore $l_{\varrho_n^\epsilon}(\gamma)\geq 1$ and $\varrho_n^{\epsilon}$ is admissible for $\G$ relative $\calK_n$.

\subsubsection{Estimating the mass of $\varrho_n^{\epsilon}$} To estimate $A(\varrho_n^\epsilon)$ note that
\begin{align*}
A(\varrho_n^{\epsilon})
& =
\int_{\mathcal{R}_n^{\epsilon} \cup \mathcal{B}_n^{\epsilon}} (\rho_n^\epsilon)^2 d\mathcal{H}^2 + \sum_{s_j\in\calK_n} \rho_{n,j}^2 \\
& =\calH^2 (\mathcal{R}_n^{\epsilon}) +\calH^2(\mathcal{B}_n^{\epsilon}) + \sum_{v_j \in \mathcal{K}_{\epsilon,n}} (\epsilon l(v_j))^2.
\end{align*}
Since $\epsilon l(v_j)$ is the side length of each of the buffer squares, we have that
$$\calH^2(B_j^{\epsilon})
= 2 (\epsilon l(v_j))^2=2\epsilon \calH^2(v_j^{\epsilon})$$ and therefore
\begin{align}\label{ineq:mass-estimate}
\begin{split}
A(\varrho_n^{\epsilon})
& =\calH^2 (\mathcal{R}_n^{\epsilon}) +(3/2)\calH^2(\mathcal{B}_n^{\epsilon}) 
= \calH^2 (\mathcal{R}_n^{\epsilon}) + 3 \epsilon \calH^2\left(\bigcup_{v_j\in \mathcal{K}_{\epsilon,n}} v_j^{\epsilon}\right)\\
&\leq \calH^2 (\mathcal{R}_n^{\epsilon}) + 3 \epsilon,
\end{split}
\end{align}
where the last inequality holds since $v_j^{\epsilon}$'s are pairwise disjoint and $\cup_j v_j^{\epsilon}\subset\mathbb{U}$.
%

To estimate $\mathcal{H}^2(\calR_n^{\epsilon})$, we first note that $\mathcal{H}^2(\mathcal{R}_0^{\epsilon})=1-\epsilon l(s_0) = 1-\epsilon r_0$. Next, assume  that for some $n\geq1$ we have $\calH^2(\mathcal{R}^{\epsilon}_{n-1}) \leq \prod_{i=1}^{n-1} (1- \epsilon r_i^2)$. From the definition of $\calR_n^{\epsilon}$ and the disjointness properties of the collars we have that
%
%
%
\begin{align*}
  \calR_n^{\epsilon} = [0,1]^2 \setminus  \bigcup_{v_j \in \calK_n'} v_j^{\epsilon} = [0,1]^2 \setminus  \bigcup_{s_i\in\calK_n} s_i^{\epsilon}.
\end{align*}
%
Next, we observe that if  $\D\in\calD_n$, $n>1$, then 
\begin{align}\label{equality:measure-of-nsss}
  \calR^{\epsilon}_n \cap \D  =(\calR^{\epsilon}_{n-1} \cap \D) \setminus s^{\epsilon}(\D),
\end{align}
where $s(\D)$ is the slit corresponding to $\D$.
Indeed, as noted above either $s^{\epsilon}(\D)$ is contained in a previously removed collar, or it does not intersect any such collar.
If $s^{\epsilon}(\D)$ is contained in a previously removed collar then by (the proof of) Lemma \ref{lemma:disjoint}, the dyadic square $\D$ is also in the complement of $\calR_{n-1}^{\epsilon}$ and both sides of (\ref{equality:measure-of-nsss}) are empty. On the other hand if $s^{\epsilon}(\D)\cap\calR_{n-1}^{\eps}\neq \emptyset$ then $s^{\epsilon}(\D)\subset\calR_{n-1}^{\epsilon}$ (again by Lemma \ref{lemma:disjoint}) and (\ref{equality:measure-of-nsss}) follows from the definition of $\calR_n^{\epsilon}$.  
%
%

From (\ref{equality:measure-of-nsss}) we have that if $\D\in\calD_n$ is such that $\mathcal{R}_{n-1}^{\epsilon}\cap \D\neq \emptyset$  then
\begin{align*}
  \calH^2(\calR^{\epsilon}_n \cap \D)=\calH^2(\calR^{\epsilon}_{n-1} \cap \D)-\calH^2(s^{\epsilon}(\D)).
\end{align*}
But
\begin{align*}
  \calH^2(s^{\epsilon}(\D))=\epsilon l(s(\D))^2 = \epsilon \left(\frac{r_n}{2^n}\right)^2 = \epsilon r_n^2 \calH^2(\D) \geq \epsilon r_n^2 \calH^2(\calR^{\epsilon}_{n-1} \cap \D),
\end{align*}
and therefore if $s(\D)$, or equivalently $\Delta$, intersects $ \calR_{n-1}^{\epsilon}$ then we have
\begin{align}\label{ineq:residual}
  \calH^2(\calR^{\epsilon}_n \cap \D)\leq(1-\eps r_n^2) \calH^2(\calR^{\epsilon}_{n-1}\cap \D).
\end{align}
Moreover, if $\D\cap \calR_{n-1}^{\epsilon} = \emptyset$ then both sides in (\ref{equality:measure-of-nsss}) are empty and (\ref{ineq:residual}) still holds with both sides being $0$. Summing (\ref{ineq:residual}) over all dyadic cubes of generation $n$ we obtain $$\calH^2(\calR^{\epsilon}_n)\leq(1-\epsilon r_n^2) \calH^2(\calR^{\epsilon}_{n-1}).$$
By induction hypothesis we have
$\calH^2(\calR^{\epsilon}_k)\leq\prod_{i=0}^k(1-\eps r_i^2),$
and therefore by (\ref{ineq:mass-estimate}) we obtain
\begin{align*}
A(\varrho_n^\epsilon) \leq \prod_{i=0}^{n}(1-\epsilon r_i^2)+ 3\epsilon.
\end{align*}
Since $\varrho_n^{\epsilon}$ is admissible for $\G$ relative $\mathcal{K}_n$ we obtain (a stronger version of) inequality (\ref{trmod}) in the case when $r_i$'s and $\epsilon$ are powers of $2$.

To prove (\ref{trmod}) in general, assume $r_i, i\geq 0,$ and $\epsilon$ are arbitrary numbers in $(0,1)$. Then there are integers $j_i\geq 1$ and $m\geq1$ such that  $2^{-j_i}\leq r_i < 2^{-j_i+1}$ and $2^{-m}\leq \epsilon < 2^{-m+1}$. Let $\epsilon'=2^{-m}$, $r_i' = 2^{-j_i}$, and let $\mathcal{K}'_n, n=0,1,\ldots,$ be the families of dyadic slits $\{s'(\Delta)\}_{\Delta\in\mathcal{D}}$, corresponding to the sequence $\{r_i'\}_{i=0}^{\infty}$, cf. Section \ref{Section:slit-carpets}. Defining the omitted, residual and buffer sets  for $\mathcal{K}_n'$ as before, we let $(\varrho_n^{\epsilon})'=((\rho_n^{\epsilon})';\rho_{n,1}',\ldots,\rho_{n,N}')$ be the mass distribution corresponding to $\mathcal{K}_n'$ defined as in (\ref{mass-distribution}). In particular, the weight corresponding to a slit $s'=s'(\Delta)\in\mathcal{K}'_n$ is either $0$ or is given by $\rho_{n,i}'=\epsilon' l(s')$. 

Next, define a new mass distribution $\zeta_{n}^{\epsilon}=(\sigma_n^{\epsilon};\{\sigma_{n,i}\})$ relative $\mathcal{K}_n$ by setting $\sigma_{n,i}$, i.e., the weight of $s_i$, to be the same as the weight of $s_i'$, and by letting $\sigma_{n}^{\epsilon}$ to be the restriction of $(\rho_{n}^{\epsilon})'$ to $\mathbb{U}\setminus K_n$. Just like above, one may see that $\zeta_n^{\epsilon}$ is admissible for $\G$ relative $\mathcal{K}_n$. Therefore,
\begin{align*}
\Mod_{\mathbb{U},\mathcal{K}_n} (\G)
\leq A((\varrho_n^{\epsilon})')
\leq \prod_{i=0}^{n}(1-\epsilon' (r_i')^2)+ 3\epsilon'.
\end{align*}
Since $\epsilon r_i^2 \leq 2\epsilon' (2r'_i)^2$ and $\epsilon'\leq \epsilon$, the last inequality implies (\ref{trmod}) in general.
%
%
%
%
%

Finally, if $\mathbf{r}\notin \ell^2$ then the first term in the right hand side of (\ref{trmod}) approaches $0$ as $n\to\infty$.  Therefore, for every $\epsilon>0$ we have that $$\limsup_{n\to\infty}\Mod_{\mathbb{U},\calK_n} (\G)\leq 3\epsilon,$$ which implies (\ref{limit:trmod0}) and completes the proof.
\end{proof}

\section{Embeddings of Slit Carpets}
\label{Section:embedding}

In this section, we prove the ``if" direction in Theorem \ref{thm:main}.
\begin{theorem}\label{thm:embedding}
If $\mathbf{r}=\{r_i\}_{i=0}^{\infty} \in \ell^2$ then there is a quasisymmetric embedding $F:\mathscr{S}=\mathscr{S}_{\mathbf{r}}\hookrightarrow\mathbb{R}^2$.
\end{theorem}

The idea is to show that there is a metric $2$-sphere $\widehat{\mathscr{S}}$ which contains $\mathscr{S}$ and is quasisymmetric to the standard sphere $\mathbb{S}^2$. The surface $\widehat{\mathscr{S}}$ will be obtained by ``gluing in" topological disks along the peripheral circles of the slit carpet $\mathscr{S}$. We will then use Bonk and Kleiner's uniformization theorem, cf. \cite{Bonk Kleiner}, to show that $\widehat{\mathscr{S}}$ is quasisymmetric to $\mathbb{S}^2$.

\subsection{Pillowcases}
For $l\in(0,1)$  consider the rectangle
$R = R(l)= [-l,l]\times[0,l]$.
Define an equivalence relation on $\partial R$  by identifying $(x,0)$ with $(-x,0)$, and $(x,1)$ with $ (-x,1)$
for $x\in[0,l]$.
%
The quotient space 
\begin{equation}\label{spillow}
\mathscr{P}=\mathscr{P}(l) = R(l) / \sim
\end{equation}
can be thought of as a ``square pillowcase" with an open ``mouth", which corresponds to the vertical sides of the rectangle $R$. For this reason we will call $\mathscr{P}$ a  \textit{square pillowcase of sidelength $l$}. The image of  a point $z \in R$ in $\mathscr{P}$ under the quotient map will be denoted by $[z]$.
We will also use the following notations:
\begin{eqnarray*}
T(\mathscr{P}) &=& \{[(0,t)]:0 \leq t \leq l\},
	\\
	L(\mathscr{P}) &=& \{[(t,0)]:0 \leq t \leq l\},
	\\
	U(\mathscr{P}) &=& \{[(t,l	)]:0 \leq t \leq l\},
\end{eqnarray*}
and will call these sets the \textit{top, lower and upper edges of $\mathscr{P}$}, respectively. Clearly, $\mathscr{P}$ is a topological disk and $\partial \mathscr{P}$ is a topological circle corresponding to the vertical sides of $R$.

As a metric space, $\mathscr{P}$ is equipped with the quotient of the Euclidean metric on $R$, cf. \cite[Section 3.1]{Burago}.

\subsection{The ``pillowcase" surface.}\label{Section:gluing} Next we show how one can glue a pillowcase to a slit of the slit carpet $\S$.
Suppose $s\subset \S$ is a slit such that $\pi(s) = \{x\} \times [a,a+l] \subset \mathrm{int}({\mathbb{U}})$. Given a point $z=(x,a+t)\in\pi(s)$ we will denote by $p_z^+$ or $p_z^-$ the preimages of $z$ in $\S$ which are closer to the right or left sides of the outer square of $\S$, respectively. Note that for the endpoints of the slit $s$, i.e., for $z=(x,a)$ and $z=(x,a+l)$ we have $p_z^{+}=p_z^-$.

Next, for a slit of length $l$ consider the mapping  
\begin{align}\label{gluefunction}
\begin{split}
g(s): \partial\mathscr{P}(l)& \to s\\
[(l,t)] &\mapsto p_{(x,a+t)}^+, \\
[(-l,t)] &\mapsto p_{(x,a+t)}^-.
\end{split}
\end{align}
Clearly $g(s)$ is a homeomorphism and is an isometry when $\mathscr{P}(l)$ is equipped with the quotient metric and $s$ with the restriction of the metric in $\S$. 

Given a slit carpet $\S$ we define the double $D\S$ of $\S$ by taking two copies of $\S$ and identifying them along the outer square, i.e., denoting by $\S_1$ and $\S_2$ the two copies of $\S$ we have
\begin{align*}
D\S = \S_1 \sqcup \S_2/\sim, 
\end{align*}
where $p_1\in\S_1$ is equivalent to $p_2\in\S_2$ if they correspond to the same point on $\partial{\mathbb{U}}$. It follows from Whyburn's theorem that as a topological space $D\S$ is homeomorphic to the Sierpi\'nski carpet.  Moreover, the path metric $d_{\S}$ naturally induces a quotient metric on $D\S$, which we will denote by $d_{D\S}$.

Let $D\mathcal{K}$ denote the collection of all slits in $D\mathscr{S}$, and  let $D\mathcal{K} = \{s_j\}_{j=0}^\infty$ be an enumeration of the slits. To each slit $s_j$ in $D\mathcal{K}$ we assign a pillowcase $\mathscr{P}_j$ of sidelength equal to $\diam(s_j)=l(s_j)$ and a gluing function $g_j=g(s_j):\partial\mathscr{P}_j\to s_j$ as defined in (\ref{gluefunction}). 


Thus, for every slit carpet we may define the  topological space $\widehat{\mathscr{S}}$ as follows. Consider the quotient space 
\begin{equation}
{\widehat{\mathscr{S}}}= \left(D\mathscr{S} \sqcup (\cup_{j=0}^{\infty} \mathscr{P}_j) \right) / \sim,
\end{equation}
obtained by gluing the pillowcase $\mathscr{P}_j$ to $D\S$ via $g_j$, i.e., for $j\geq0$, if $x\in \partial\mathscr{P}_j$ then we have that $x\sim g_j(x)$. Thus, we cover every slit with a square pillowcase by gluing its boundary with the corresponding slit isometrically.

Note that $\widehat{\mathscr{S}}$ is homeomorphic to $\mathbb{S}^2$ since every $\mathscr{P}_i$ is a topological disk and $D\mathscr{S}$ is homeomorphic to $\mathbb{S}_{1/3}$ by Whyburn's Theorem \ref{metricsierpinski}. 


The space $\widehat{\mathscr{S}}$ can be equipped with a natural  metric studied by Ha\"issinsky in  \cite{Haisinsky}. First, define a quasimetric $\tau$ on $\widehat{\mathscr{S}}$ by setting
\begin{align*}
\begin{split}
\tau(p,q) = 
\begin{cases}
d_{D\S}(p,q), & \mbox{ if } p,q\in D\S,\\
d_i(p,q), & \mbox{ if } p,q\in \mathscr{P}_i, i\geq 1,\\
\inf_{\zeta\in s_i} \{d_{D\S}(p,\zeta) + d_i(\zeta,q)\}, & \mbox{ if } p\in D\S, q \in \mathscr{P}_i, \\
\inf_{\zeta\in s_i, \xi\in s_j} \{d_i(p,\zeta) + d_{D\S}(\zeta,\xi)+d_j(\xi,q)\}, & \mbox{ if } p\in \mathscr{P}_i, q \in \mathscr{P}_j, i\neq j, \\
\end{cases}
\end{split}
\end{align*}
where $d_i, i\geq 1,$ denotes the metric on $\mathscr{P}_i$.
Furthermore, for $p,q\in\widehat{\mathscr{S}}$ let 
\begin{align}
d_{\widehat{\mathscr{S}}}(p,q) = \inf \sum_{k=0}^{N-1} \tau(\zeta_k,\zeta_{k+1}),
\end{align}
where the infimum is taken over all finite chains $\zeta_0,\ldots,\zeta_N$  in $\widehat{\mathscr{S}}$ such that $\zeta_0=p$, $\zeta_N=q$. By Theorem 2.2 in \cite{Haisinsky}, $d_{\widehat{\mathscr{S}}}$ is a metric provided the mappings $g_i$ are uniformly quasisymmetric and $\diam_{d_i}\mathscr{P}_i \leq C \diam_{d_i} \partial \mathscr{P}_i$, for all $i\geq 1$. Since in our case the mappings $g_i$ are all isometries, and the inequality above holds with $C=\sqrt{2}$, it follows that $d_{\widehat{\mathscr{S}}}$ is indeed a metric. Moreover, by \cite{Haisinsky} the restriction of $d_{\widehat{\mathscr{S}}}$ to the slit carpet $\S\subset\widehat{\mathscr{S}}$ is comparable to $\tau$, or equivalently, is bi-Lipschitz to $d_{\S}$. Therefore, to show that $(\S,d_{{\mathscr{S}}})$ quasisymmetrically embeds into the plane (or $\mathbb{S}^2$) it is enough to show that $\widehat{\mathscr{S}}$ is quasisymmetric to $\mathbb{S}^2$.
For this we will need the following well known uniformization result of Bonk and Kleiner.
\begin{theorem}[Bonk, Kleiner, \cite{Bonk Kleiner}]\label{quasisphere}
	Let $X$ be an Ahlfors $2$-regular compact connected metric space homeomorphic to $\mathbb{S}^2$. Then $X$ is quasisymmetric to $\mathbb{S}^2$ if and only if $X$ is linearly locally connected.
\end{theorem}

%




Recall that a metric space $(X, d)$ is called \emph{linearly locally connected (or LLC)} if there exists a constant $\lambda \geq 1$ so that for every $z\in X$ and $r>0$ the following conditions hold:
\begin{itemize}
\item[($LLC_1$)] If $x,y \in B(z, r)$, then there exists
a continuum $E \subset B(z, \lambda r)$ containing $x$ and $y$.
\item[($LLC_2$)] If $x,y \notin B(z, r)$, then there exists
a continuum $E \subset X \backslash B(z, r/\lambda)$ containing $x$ and $y$.
\end{itemize}

Thus, by Theorem \ref{quasisphere}, to complete the proof we need to show that $\widehat{\mathscr{S}}$ is LLC and Ahlfors $2$-regular.

\begin{remark}\label{remark:BK}
Theorem \ref{quasisphere} is quantitative in the sense that the quasisymmetric mapping $f:X\to\mathbb{S}^2$ is $\eta$-quasisymmetric with $\eta$ depending only on the  Ahlfors regularity and LLC constants for $X$.
\end{remark}

By \cite[Theorem 2.6.2]{Haisinsky} the metric sphere $\widehat{\mathscr{S}}$ is LLC provided $D\S$ and all $\mathscr{P}_i$, $i\geq 1$ are uniformly LLC. Since $\mathscr{P}_i$'s are all uniformly LLC (with $\lambda=1$) it is enough to show that $D\S$ is LLC.

\begin{lemma}\label{lemma:llc}
The double $D\S$ of the slit carpet $\S=\S_{\mathbf{r}}$ is $LLC$.
\end{lemma}
\begin{proof}
Note that if $x\in B(z,r)$  and $\g_{xz}$ denotes a length minimizing curve connecting $x$ and $z$, then for every $p\in\g_{xz}$ we have $d_{D\S} (z,p)\leq d_{D\S}(z,x)$ and therefore $\g_{xz}\subset B(z,r)$. Therefore if $x,y\in B(z,r)$ then $\g_{zx}\cup\g_{zy}\subset B(x,r)$ is a continuum connecting $x$ and $y$. Therefore $D\S$ is $LLC_1$ with $\lambda=1$.

To show that $D\S$ is $LLC_2$ let $x,y\in D\S\setminus B(z,r)$, where  $2^{-n-1}\leq r < 2^{-n}$. Let
\begin{align*}
T' = \bigcup_{\substack{\D\in \mathcal{D}_{n+3} \\ T_{\D}\cap B(z,2^{-(n+3)}) \neq \emptyset}} T_{\D}
\end{align*}
where, as before $T_{\D}=\overline{\pi^{-1}(\mathrm{int}(\D))}$  denotes a ``dyadic square" in $\S$ corresponding to  some dyadic square $\D\subset\mathbb{U}$. Note that, since $\diam_{\S}{T_{\D}} \leq 2\cdot 2^{-n-3}$ for $\D\in\mathcal{D}_{n+3}$, we have that for every $p\in T'$ the following inequalities hold:
\begin{align*}
d_{D\S}(p,z)\leq 2^{-(n+3)}+\diam{T_{\D}} \leq 3 \cdot 2^{-(n+3)}\leq\frac{3}{4}r.
\end{align*}
Therefore, $$B\left(z,\frac{r}{8}\right)\subset B\left(z,\frac{1}{2^{n+3}}\right) \subset T'\subset B\left(z,\frac{3}{4}r\right).$$ Finally, since $x,y \in D\S\setminus \partial T'$ there is a continuum connecting $x$ and $y$ without intersecting $B(z,r/8)$. Indeed, if $x$ and $y$ belong to the same ``dyadic" square $T_{\D}$ for some $\D\in \mathcal{D}_{n+3}$ then there is a curve $\g_{xy}\subset T_{\D}$ connecting $x$ and $y$, since $T_{\D}$ is path connected. On the other hand, if $x\in T_{\D}$ and $y\in T_{\D'}$, we can first connect $x$ and $y$ to the ``outer squares" of $T_{\D}$ and $T_{\D'}$, respectively, and then we may connect these outer squares to each other through the preimages of the grids $\tilde{\Pi}_{n+3}$, cf. Section \ref{Section:necessary-condition}, without intersecting $\mathrm{int}(T')$. This gives a continuum $\g_{x,y}\subset D\S\setminus \mathrm{int}(T')$ connecting $x$ and $y$ in general. Therefore $\g_{xy}\subset D\S \setminus B(z,r/8)$ and $D\S$ is $LLC_2$.
%
%
%
%
%
%
%
\end{proof}

\begin{lemma}\label{lemma:Ahlfors-reg}
If $\mathbf{r}\in\ell^2$ then $\widehat{\mathscr{S}}$ is Ahlfors
$2$-regular.
\end{lemma}

\begin{proof}
Note that it is enough to show that the space 
$\mathscr{D}=\S\sqcup(\cup_{s_j\subset\S} \mathscr{P}(s_j))/\sim$ is Ahlfors regular. Indeed, $\widehat{\mathscr{S}}$ can be obtained by gluing two copies of $\mathscr{D}$ along the outer square by the identity, and therefore if $\mathscr D$ is Ahlfors 2-regular with constant $C$ then $\widehat{\mathscr{S}}$ is Ahlfors regular with  constant $2C$. 

Below we use the same notation $T=T_{\D}\subset\S$ as above for the dyadic squares in $\S$. Moreover, for a dyadic square $\D\in\mathcal{D}_n$ in $\mathbb{U}$ we let $\tilde{T}=\tilde{T}_{\D}$ denote the portion of $\mathscr D$ ``over" $T$, i.e.,
\begin{align*}
\tilde{T}:=\tilde{T}_{\D} = (T_{\D}\sqcup \bigcup_{s_j\subset T_{\D}} {\mathscr P(s_j)})/\sim,
\end{align*}
where $\sim$ is the same ``gluing" equivalence relation as before.

Next, suppose $\D$ is a dyadic square of generation $n\geq 1$.
Then, by Lemma  \ref{slitmeasure}, there is a constant $C\geq 1$  which does not depend on $n$, so that the following inequalities hold:
\begin{align}\label{est:measure1}
\begin{split}
\calH^2(\tilde{T}_{\D}) & = \calH^2 (T_{\D}) + \sum_{s_j\subset T_{\D}} \calH^2(\mathscr{P}(s_j)) \\
& \leq C(2^{-n})^2 + \sum_{k\geq n}\,\left[ \sum_{\substack{s(\D')\subset T_{\D} \\ \D'\in\mathcal{D}_k}} 2 l(s(\D'))^2\right].
\end{split}
\end{align}
The number of generation $k\geq n$ slits (or equivalently dyadic subsquares) contained in $\D$ is equal to $4^{k-n}$. Therefore, since $l(s(\D'))=r_k 2^{-k}$ for $\D'\in \mathcal{D}_k$, the following equality holds for every $k\geq n$:
\begin{align}\label{est:measure2}
\sum_{\substack{s(\D')\subset T_{\D} \\ s_j\subset\mathcal{D}_k}} l(s(\D'))^2 = 
4^{k-n} (r_k2^{-k})^2.
\end{align}
%
%
Hence, combining (\ref{est:measure1}) and (\ref{est:measure2}) we obtain
\begin{align*}
\calH^2(\tilde{T}_{\D})
& \leq C 4^{-n} + \sum_{k\geq n} 4^{k-n} (r_k^2 4^{-k})
 = 4^{-n} (C + \sum_{k\geq n} r_k^2).
\end{align*}
Since $2^{-n}\leq \diam T_{\D} \leq 2^{-n+1}$ we obtain that for every $\D\in\mathcal D$ the following inequalities hold:
\begin{align*}
\frac{1}{4C}(\diam T_{\D})^2\leq \calH^2(\tilde{T}_{\D}) \leq C_1 (\diam T_{\D})^2,
\end{align*}
where $C_1=C+\sum_{k=1}^{\infty} r_i^2,$ with $C$ being the constant from Lemma \ref{slitmeasure}.

Now, if $x\in\S$ and $2^{-n-1}\leq r < 2^{-n}$ then considering a dyadic square $T_{\D}$ for some $\D\in \mathcal{D}_{n+3}$  such that $B(x,r/8)\cap T_{\D}\neq \emptyset$, we have (like in Lemma \ref{lemma:llc}) $T_{\D}\subset B(x,r)$ and 
\begin{align}\label{est:ahlfors-lower-inS}
\calH^2(B(x,r)) \geq \calH^2(T_{\D}) \geq \frac{1}{4C}(\diam T_{\D})^2 \geq \frac{1}{4C}\left(\frac{r}{2^3}\right)^2 = \frac{r^2}{2^8C}.
\end{align}
On the other hand, since $\pi(B(x,r))\subset B(\pi(x),r)$, there are at most $9$ dyadic squares of generation $n$ intersecting $B(\pi(x),r)$ such that their union is a Euclidean square in $\mathbb{U}$. It follows that there are at most $9$ dyadic squares $\D_1,\ldots,\D_9\in \mathcal{D}_{n}$ such that $B(x,r) \cap \tilde{T}_{\D_i}\neq \emptyset, i=1,\ldots 9.$ Let 
$$\tilde T=\cup_{i=1}^{9}\tilde{T}_{\D_i}.$$ Then, we have
\begin{align}\label{est:ahlfors1}
\begin{split}
\calH^2(B(x,r)\cap \tilde T)& \leq \sum_{i=1}^9 \calH^2 (\tilde{T}_{\D_i}) \leq 9 C_1 (\diam(\tilde T_{\D_i}))^2 \leq 9 C_1  (2\cdot 2^{-n})^2\\
& \leq 9\cdot 2^4C_1 r^2. 
\end{split}
\end{align}
Next, if $y\in B(x,r)\setminus \tilde T$ then $y$ belongs to a pillowcase $\mathscr{P}(s_j)$ over a slit $s_j$ of generation $\leq n-1$, thus $l(s_j)\geq 2^{-n+1} > r$. Note that if $z\in\partial \mathscr P(s_j)$ is the closest point in $\mathscr P(s_j)$ to $x\in\S$, we have that $\mathscr P(s_j)\cap B(x,r)$ is contained in $\mathscr P(s_j)\cap B(z,r)$. Therefore
\begin{align*}
\calH^2(B(x,r)\cap \mathscr P(s_j)) \leq \calH^2(B(z,r)\cap \mathscr P(s_j))\leq\pi r^2/2,
\end{align*}
since $z\in \partial \mathscr P(s_j)$.

On the other hand, from the construction of $\tilde{T}$ it follows that there are at most $8$ such ``large pillowcases" $\mathscr P(s_j)$'s intersecting $\tilde{T}$, (two for every ``vertical curve" containing a vertical side of some $\tilde T_{\D_j}\subset \tilde T$). 
Therefore,
\begin{align}\label{est:ahlfors2}
\calH^2(B(x,r)\setminus \tilde T) \leq 4 \pi r^2.
\end{align}

Combining (\ref{est:ahlfors-lower-inS}), (\ref{est:ahlfors1}) and (\ref{est:ahlfors2}) we obtain that for every $x\in\S$ and $0<r\leq \diam \S$ the following holds:
\begin{align}\label{ineq:Ahlfors-reg-inS}
\calH^2(B(x,r)) \asymp r^2.
\end{align}

Finally, for $x\in\mathscr P_j$ there are three possibilities:
\begin{itemize}
\item[(1).] If $r<l(s_j)$ then there is a point $y\in B(x,r)$ such that $B(y,r/2)\subset \mathscr P_j$ and therefore $\mathcal{H}^2(B(x,r)\gtrsim r^2$. To get the upper estimate, first note that if $B(x,r)\cap s_j =\emptyset$ then $\calH^2(B(x,r))\leq \pi r^2$. On the other hand, if there exists $y\in B(x,r)\cap s_j$, then $B(x,r)\subset B(y,2r)$ and therefore by (\ref{ineq:Ahlfors-reg-inS}) we have $\calH^2(B(x,r))\lesssim r^2$.
\item[(2).] If $l(s_j)\leq r \leq 2l(s_j)$ then  $$\calH^2(B(x,r)) \geq \calH^2(B(x,\frac{r}{2}))\gtrsim r^2,$$
by part (1), since $r/2<l(s_j)$. On the other hand, since $\widehat{\mathscr{S}}$ is easily seen to be a metric doubling space, every ball $B(x,r)$ can be covered by $N$ balls $B_i=B(x_i,r/2)$ of radius $r/2<l(s_j)$, with $N$ independent of $x$. Therefore,  $\calH^2(B(x,r))\leq \sum_{i=1}^N \calH^2(B(x_i,r/2))\lesssim r^2$ by (\ref{ineq:Ahlfors-reg-inS}) and part (1).
\item[(3).] If $r> 2l(s_j) > \diam(\mathscr P_j)$ then there is a point $y\in B(x,r)\cap s_j$ such that 
\begin{align*}
B(y,\frac{r}{2})\subset B(x,r)\subset B(y,2r).
\end{align*} 
Therefore $\calH^2(B(x,r))\asymp r^2$ by (\ref{ineq:Ahlfors-reg-inS}). \qedhere
\end{itemize}
\end{proof}

\begin{proof}[Proof of Theorem  \ref{thm:embedding}]
Combining Lemma \ref{lemma:llc} and Lemma \ref{lemma:Ahlfors-reg} with Theorem \ref{quasisphere} we obtain a quasisymmetric mapping $g:\widehat{\mathscr{S}}\to\mathbb{S}^2$. By \cite{Haisinsky} $d_{\widehat{\mathscr S}}$ is comparable to the semi-metric $\tau$ (cf. Section \ref{Section:gluing}) when restricted to $\S\subset\widehat{\S}$. Since $\tau$ on $\S$ is equal to $d_{\S}$, it follows that  $id:(\S,d_{\S})\to(\S,d_{\widehat{\mathscr{S}}} |_{\S})$ is a bi-Lipschitz map. Therefore $f=g\circ id : \S\to\mathbb{S}^2$ is a quasisymmetric embedding.
\end{proof}

\section{Proof of Theorem \ref{thm:slit-carpets-TLP}} \label{Section:proof}
\begin{proof}
Let $\S=\S_{\mathbf{r}}$ be a dyadic slit carpet corresponding to $\mathbf{r}=\{r_i\}_{i=1}^{\infty}$. Suppose there is a quasisymmetric embedding of $\S$ into the plane. By Theorem \ref{thm:main} (or Corollary \ref{thm:necessity}) we have that $\sum r_i^2<\infty$. Let $\widehat{\S}_n$ be the $n$-th pillowcase surface obtained by gluing in pillowcases to the slits of the double of $\S_n$ like in Section \ref{Section:embedding}. Just like above $\widehat{\S}_n$ is Ahlfors $2$-regular and LLC with constants independent of $n$ (in fact with the same constants that work for $\widehat{\S}$ ). Since Bonk-Kleiner theorem is quantitative, see Remark \ref{remark:BK}, we have that for every $n\geq 1$ there is an $\eta$ - quasisymmetric maps $g_n:\widehat{\S}_n\to\mathbb{S}^2$ for a fixed $\eta$. Therefore there  are uniformly quasisymmetric embeddings $f_n:=g_n\circ id_n :\S_n\to\mathbb{S}^2$, since the inclusion maps $id_n:\S_n\to \widehat{\S}_n$ are uniformly bi-Lipschitz. Hence, $S_n$'s are $\Psi$ - TLP for a fixed $\Psi$ by Theorem \ref{thm:finitely-connected-embedding}, and the slit carpet $\S$ is TLP.

Conversely, if $\S$ is TLP then denoting by $L$ and $R$ the left and right vertical sides of the unit square and observing that 
$\delta(L,R) = 1$ we have
\begin{align*}
\mod_{\mathbb{U},\calK_n} \G(L,R) \geq \Psi(\Delta(L,R))\geq \Psi(1)>0.
\end{align*}
From Lemma \ref{lemma:transmod=0} it follows that $\sum r_i^2<\infty$.  Hence, by Theorem \ref{thm:main} there is a quasisymmetric embedding of $\S$ into $\mathbb{S}^2$.
\end{proof}

\end{document}